\documentclass[11pt,reqno]{amsart}
\usepackage{amsaddr}
\usepackage[pdftex]{graphicx}
\usepackage{pgfplots}
\usepackage{orcidlink}
\usepackage{geometry}
\usepackage{stix}
\usepackage{subcaption}

\usepackage{caption}
\usepackage{amsmath,amsthm}
\usepackage{cite}
\usepackage{hyperref}
\hypersetup{
    colorlinks=true, 
    linktoc=all,     
    allcolors=blue,  
}
\usepackage{url}
\newtheorem{theorem}{Theorem}[section]
\newtheorem{lemma}{Lemma}[section]
\newtheorem{proposition}{Proposition}[section]
\newtheorem{corollary}{Corollary}[section]

\theoremstyle{definition}
\newtheorem{definition}{Definition}[section]

\newtheorem{remark}{Remark}[section]
\numberwithin{equation}{section}
\pgfplotsset{compat=1.18}
\definecolor{wewdxt}{rgb}{0.45,0.45,0.45}
\definecolor{zzttqq}{rgb}{0.6,0.2,0.}
\definecolor{qqwuqq}{rgb}{0.,0.5,0.}
\definecolor{uququq}{rgb}{0.25,0.25,0.25}
\begin{document}
\pagenumbering{arabic}

\title[Siebeck--Marden Theorems and Symmetric Parametrizations]{Siebeck--Marden Theorems and Symmetric Parametrizations of Cyclic Polygons}
\author{Mohammad Hassan Murad\orcidlink{0000-0002-8293-5242}}
\address{Department of Mathematics\\
The University of Texas at Arlington, Arlington, TX, USA}
\email{mohammad.murad2@uta.edu}

\begin{abstract}
The classical Blaschke-product approach provides an elegant description of triangles inscribed in a circle and circumscribed about an ellipse. Motivated by the Siebeck--Marden theorem, we derive a symmetric parametrization for cyclic 
$p$-gons circumscribed about a Siebeck--Marden curve of class $p-1$, recovering the Blaschke product parametrization as the triangular case. We then use this parametrization to establish several geometric invariance results, including a characterization of when the sum of the squares of all sides and diagonals is invariant. We  revisit Cayley's criterion for 
3- and 4-Poncelet pairs and obtain a complete classification of the associated central conics solely in terms of their foci. In particular, we show that the formulas for the lengths of the major axis obtained under the assumption on the foci of the conic lying inside the circumcircle remain valid when one focus or both foci lie outside the circumcircle, thereby extending the classical theory from ellipses to all admissible central conics. Several geometric properties of the associated cyclic quadrilaterals are also obtained.
\end{abstract}



\maketitle{}

\section{Introduction}\label{sec:intro}
\noindent
The interplay between Poncelet geometry, finite Blaschke products, and numerical ranges  has been an active area of research for several decades. Of particular importance is the surprising connection between finite Blaschke products of degree three and families of triangles inscribed in a circle and circumscribed about an ellipse.

To illustrate this approach, let $\mathbb D$ denote the open unit disk in $\mathbb C$ and let
$\mathbb T=\partial\mathbb D$ be the unit circle. A finite Blaschke product of degree three is defined by
\[
B(z)=
z\frac{z-a_1}{1-\overline{a_1}z}
 \frac{z-a_2}{1-\overline{a_2}z},
\qquad a_1,a_2\in\mathbb D.
\]
For every $\lambda\in\mathbb T$, the equation
\[
    B(z)=\lambda
\]
has three solutions $z_1,z_2,z_3$ which determine a triangle $\triangle z_1z_2z_3$ inscribed in the unit circle $\mathbb T$.

A classical theorem of Daepp, Gorkin, and Mortini \cite{daepp2002ellipses} (see also \cite{DaeppGorkinshaffervoss2018}) states that the sides of the triangle associated with a degree-3 Blaschke product are tangent to an ellipse with foci $a_1,a_2\in\mathbb D$ whose major axis has length
\begin{equation}\label{eq:majoraxisell3}
    |1-\overline{a_1}a_2|.
\end{equation}
Consequently, Blaschke products provide a natural analytic parametrization of Poncelet triangles, allowing many geometric questions to be translated into algebraic relations involving the foci $a_1,a_2$ and the unimodular parameter $\lambda$. This approach has led to numerous results on geometric invariants of families of triangles circumscribed about ellipses inscribed in a common circumcircle; see, for example,
\cite{Helmanetal.2022,Helmanetal.2023,Garciaetal.2026}. 

Using a similar approach---Blaschke product of degree 4---Fujimura \cite{Fujimura2013} subsequently obtained an analogous parametrization for cyclic quadrilaterals circumscribed about an ellipse whose foci $a_1,a_2$ lie inside the circumcircle and proved that the major axis of the ellipse has length
\begin{equation}\label{eq:majoraxisell4}
|1-\overline{a_1}a_2|
\sqrt{\frac{2-|a_1|^2-|a_2|^2}{1-|a_1|^2|a_2|^2}}. 
\end{equation}
A fundamental feature of the Blaschke product approach is the assumption that
\[
|a_1|<1,
\qquad
|a_2|<1,
\]
which guarantees that the associated conic with the 3- and 4-Poncelet pairs is an ellipse contained in the circumcircle (Figure \ref{fig:ellell3pons(A)}). From the viewpoint of Poncelet geometry, however, this is only one of the possible focal configurations. It excludes central conics whose one focus or both foci lie outside the circumcircle. Recent work \cite{Dragovic-Murad2026} (see also \cite{Dragovic-Radnovic2024}) shows that these configurations arise naturally in the Poncelet theory. See Figures \ref{fig:ellell3pons(B)}--\ref{fig:ellell3pons(C)}.

\begin{figure}
  \begin{subfigure}[b]{0.45\textwidth}
    \centering
\begin{tikzpicture}[scale=2.5]
\clip(-1.2,-1.2) rectangle (1.2,1.2);
\draw [line width=1.pt,color=gray] (0.,0.) circle (1.cm);
\draw [rotate around={47.03091423685805:(-0.10864102302989269,0.0959646526091253)},line width=1.pt,color=red] (-0.10864102302989269,0.0959646526091253) ellipse (0.5931004472932491cm and 0.38943652776176013cm);
\draw [line width=1.pt,color=blue] (-0.18268085589958466,0.983172266130303)-- (-0.8382385956087126,-0.5453036372810411);
\draw [line width=1.pt,color=blue] (-0.8382385956087126,-0.5453036372810411)-- (0.9915802265023382,-0.12949383927265212);
\draw [line width=1.pt,color=blue] (0.9915802265023382,-0.12949383927265212)-- (-0.18268085589958466,0.983172266130303);
\begin{scriptsize}
\draw [fill=wewdxt] (0.,0.) circle (0.6pt);
\draw[color=wewdxt] (0.0328643396112826,0.08290416268653264) node {$O$};
\draw[color=wewdxt] (-0.7150897315851618,0.7959537105604741) node {$\mathbb T$};
\draw [fill=wewdxt] (-0.18268085589958466,0.983172266130303) circle (0.7pt);
\draw[color=wewdxt] (-0.18,1.0801762576151221) node {$z_1$};
\draw [fill=black] (-0.4135452259255244,-0.23135897697006552) circle (0.6pt);
\draw[color=black] (-0.34,-0.16) node {$a_1$};
\draw [fill=black] (0.19626317986573905,0.4232882821883161) circle (0.6pt);
\draw[color=black] (0.16,0.33) node {$a_2$};
\draw[color=red] (-0.46577170785301364,0.08290416268653264) node {$\mathcal D$};
\draw [fill=wewdxt] (-0.8382385956087126,-0.5453036372810411) circle (0.7pt);
\draw[color=wewdxt] (-0.92,-0.6) node {$z_2$};
\draw [fill=wewdxt] (0.9915802265023382,-0.12949383927265212) circle (0.7pt);
\draw[color=wewdxt] (1.1,-0.15) node {$z_3$};
\end{scriptsize}
\end{tikzpicture}
    \caption{The foci of $\mathcal{D}$ lie inside $\mathbb T$.}
    \label{fig:ellell3pons(A)}
\end{subfigure}
  \begin{subfigure}[b]{0.45\textwidth}
    \centering
\begin{tikzpicture}[scale=2]
\clip(-1.5,-1.5) rectangle (2,1.5);
\draw [line width=1.pt,gray] (0.,0.) circle (1.cm);
\draw [rotate around={30.96375653207352:(0.25,-0.25)},line width=1.pt,color=red] (0.25,-0.25) ellipse (1.5811388312175827cm and 0.6123724386222013cm);
\draw [line width=1.pt,domain=-1.25:-0.6,lightgray] plot(\x,{(-1.3666139290970083-1.5341662006726962*\x)/-0.4453283004437236});
\draw [line width=1.pt,domain=-1.0:1.7,lightgray] plot(\x,{(--1.3896713737525987-0.01759274556816859*\x)/1.904857468062226});
\draw [line width=1.pt,color=blue] (-0.6772039690466689,0.735795341319472)-- (-0.965920388236309,-0.2588393393389382);
\draw [line width=1.pt,color=blue] (-0.965920388236309,-0.2588393393389382)-- (0.6906785018479058,0.7231619508001873);
\draw [line width=1.pt,color=blue] (0.6906785018479058,0.7231619508001873)-- (-0.6772039690466689,0.735795341319472);
\begin{scriptsize}
\draw [fill=wewdxt] (0.,0.) circle (0.7pt);
\draw[color=wewdxt] (0.05070647863465335,0.11395926323489103) node {$O$};
\draw[color=black] (-0.41459763982998493,1.058667624966127) node {$\mathbb T$};
\draw [fill=black] (-1.,-1.) circle (0.7pt);
\draw[color=black] (-0.9,-0.92) node {$a_1$};
\draw [fill=black] (1.5,0.5) circle (0.7pt);
\draw[color=black] (1.41,0.38) node {$a_2$};
\draw[color=red] (-0.65,-0.27) node {$\mathcal D$};
\draw [fill=wewdxt] (-0.6772039690466689,0.735795341319472) circle (0.7pt);
\draw[color=wewdxt] (-0.74,0.86) node {$z_1$};
\draw [fill=wewdxt] (-0.965920388236309,-0.2588393393389382) circle (0.7pt);
\draw[color=wewdxt] (-1.09,-0.3) node {$z_2$};
\draw [fill=wewdxt] (0.6906785018479058,0.7231619508001873) circle (0.7pt);
\draw[color=wewdxt] (0.73,0.83) node {$z_3$};
\end{scriptsize}
\end{tikzpicture}
    \caption{The foci of $\mathcal{D}$ lie outside $\mathbb{T}$.}
    \label{fig:ellell3pons(B)}
\end{subfigure}
  \begin{subfigure}[b]{0.45\textwidth}
    \centering
\definecolor{xdxdff}{rgb}{0.5,0.5,1.}
\begin{tikzpicture}[scale=2]
\clip(-2,-2) rectangle (3,2.5);
\draw [line width=1.pt,gray] (0.,0.) circle (1.cm);
\draw [samples=50,domain=-0.99:0.99,rotate around={39.8055710922652:(0.25,0.375)},xshift=0.25cm,yshift=0.375cm,line width=1.pt,color=red] plot ({0.883883478235715*(1+(\x)^2)/(1-(\x)^2)},{0.4145780950580171*2*(\x)/(1-(\x)^2)});
\draw [samples=50,domain=-0.99:0.99,rotate around={39.8055710922652:(0.25,0.375)},xshift=0.25cm,yshift=0.375cm,line width=1.pt,color=red] plot ({0.883883478235715*(-1-(\x)^2)/(1-(\x)^2)},{0.4145780950580171*(-2)*(\x)/(1-(\x)^2)});
\draw [line width=1.pt,domain=-1.2:3.5,lightgray] plot(\x,{(--1.6137566140999264--0.5797844868683809*\x)/2.945291874537984});
\draw [line width=1.pt,domain=-1.2:0.5,lightgray] plot(\x,{(-0.3304765837104118-0.5495063536511812*\x)/0.4969736286455214});
\draw [line width=1.pt,domain=0.2:1.2,lightgray] plot(\x,{(-0.8827809881922324--1.6533245900836944*\x)/0.45396842724236075});
\draw [line width=1.pt,color=blue] (-0.9311601169298341,0.36461052732911803)-- (0.269525053806589,-0.9629933776358773);
\draw [line width=1.pt,color=blue] (0.269525053806589,-0.9629933776358773)-- (0.7234934810489497,0.690331212447817);
\draw [line width=1.pt,color=blue] (0.7234934810489497,0.690331212447817)-- (-0.9311601169298341,0.36461052732911803);
\begin{scriptsize}
\draw [fill=wewdxt] (0.,0.) circle (0.7pt);
\draw[color=wewdxt] (0.0391164331436456,0.10489268918362266) node {$O$};
\draw[color=black] (-0.5481280247583584,0.9339436885746862) node {$\mathbb T$};
\draw [fill=black] (-0.5,-0.25) circle (0.7pt);
\draw[color=black] (-0.58,-0.35) node {$a_1$};
\draw [fill=black] (1.,1.) circle (0.7pt);
\draw[color=black] (1.1,1.1) node {$a_2$};
\draw[color=red] (1.2423918419820656,2.0105585280616927) node {$\mathcal D$};
\draw [fill=wewdxt] (-0.9311601169298341,0.36461052732911803) circle (0.7pt);
\draw[color=wewdxt] (-1.07,0.42) node {$z_1$};
\draw [fill=wewdxt] (0.269525053806589,-0.9629933776358773) circle (0.7pt);
\draw[color=wewdxt] (0.31,-1.14) node {$z_2$};
\draw [fill=wewdxt] (0.7234934810489497,0.690331212447817) circle (0.7pt);
\draw[color=wewdxt] (0.84,0.79) node {$z_3$};
\end{scriptsize}
\end{tikzpicture}
    \caption{A focus of $\mathcal D$ lies inside $\mathbb T$ and the other lies outside.}
    \label{fig:ellell3pons(C)}
    \end{subfigure}
    \caption{The triangle $\triangle z_1z_2z_3$ is inscribed in a unit circle $\mathbb{T}$ and circumscribed about a central conic $\mathcal{D}$ with foci $a_1,a_2 \in \mathbb C$.}
    \label{fig:ellell3pons}
\end{figure}
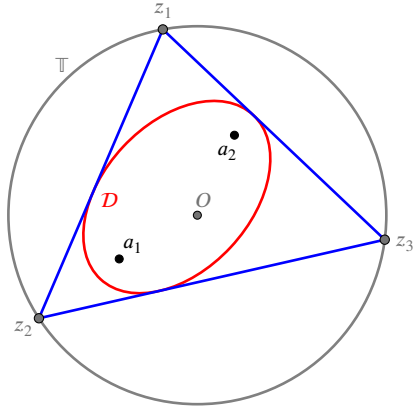
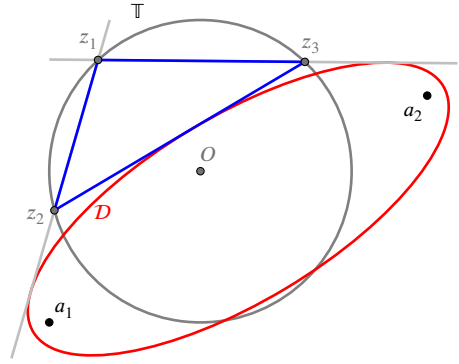
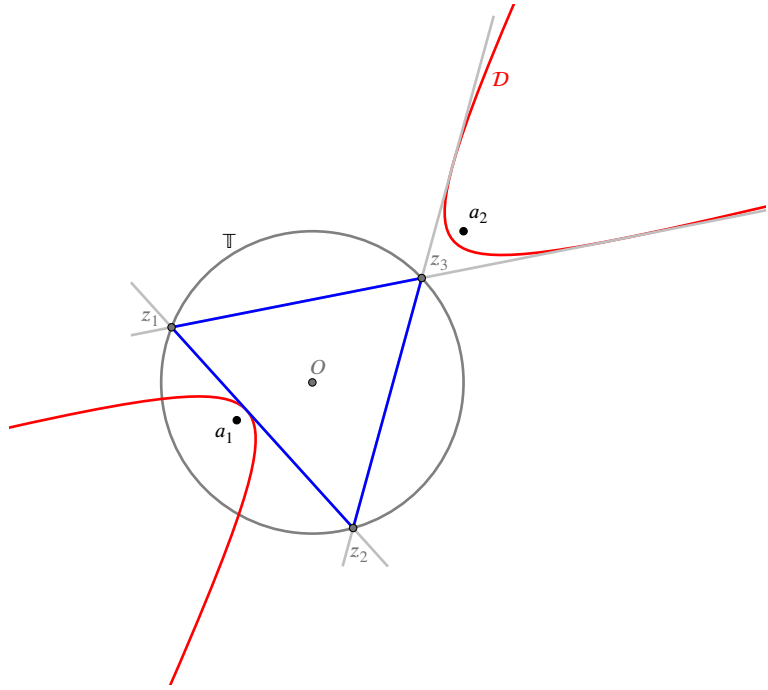

The first goal of this paper is to remove the restriction that the foci lie inside the circumcircle. Using Cayley's criterion, we characterize all central conics forming $3$- and $4$-Poncelet pairs with the unit circle solely in terms of their focal configurations. In particular, we prove that the major-axis length formulas \eqref{eq:majoraxisell3} and \eqref{eq:majoraxisell4} remain valid for every admissible pair of foci, regardless of whether the corresponding conic is an ellipse or a hyperbola. Thus, the classical formulas extend uniformly to both the elliptic and hyperbolic settings.

The second goal of the present paper is to place these parametrizations into a broader geometric framework. Motivated by the Siebeck--Marden theorem, we derive a symmetric parametrization for cyclic $p$-gons circumscribed about an algebraic curve of class $p-1$ with $p-1$ real foci. We call these curves the \emph{Siebeck--Marden curve of class $p-1$}. The classical Blaschke-product parametrization is recovered as the triangular case, providing a geometric interpretation of the Blaschke-product formalism and extending it naturally to arbitrary finite degree.

As applications of this parametrization, we establish several geometric invariants of Poncelet families. In particular, we characterize when the sum of the squares of all sides and diagonals of a cyclic $p$-gon is invariant, obtaining a necessary and sufficient condition in terms of the circumcenter and the foci of the associated Siebeck--Marden curve. We also provide a short proof and extension of a recent theorem \cite[Theorem 2.1]{Celiketal2026} on the total area of power circles and derive several new results for cyclic quadrilaterals, including focal criteria for orthodiagonality and relations involving orthoptic and auxiliary circles.

The paper is organized as follows. In Section \ref{sec:Cayley} we review Cayley's criterion for 3- and 4- Poncelet pairs of conics and derive generalized Chapple--Euler and Cayley--Fuss relations for circle--conic pairs. Section \ref{sec:siebeckmarden} develops the symmetric parametrization of cyclic polygons via the Siebeck--Marden theorem. The final section is devoted to applications of the symmetric parametrization obtained in the preceding section.

\section{Cayley's Criterion and Focal Characterizations of Circle–Conic Poncelet Pairs
}\label{sec:Cayley}
\subsection{Cayley's Criterion}
We begin with Cayley's criterion for determining whether two conics form an $n$-Poncelet pair. Throughout this section, we identify a conic with the corresponding $3\times3$ symmetric matrix defining it.

\begin{theorem}[Cayley 1853, \cite{Cayley1853a,Cayley1853b}]\label{thm.2.1}
Let $\mathcal{C}_1$ and $\mathcal{C}_2$ be two distinct conics in $\mathbb{CP}^2$ in general position, and
\[
\sqrt{\det(\lambda\mathcal{C}_1+\mathcal{C}_2)}=A_0+A_1\lambda+\cdots +A_n \lambda^n+\cdots.
\]
Then $(\mathcal{C}_1,\mathcal{C}_2)$ is an $n$-Poncelet pair if and only if
\begin{align*}
\begin{vmatrix}
    A_2 & \cdots & A_{m+1}\\
    \vdots & \ddots & \vdots \\
     A_{m+1} & \cdots & A_{2m}
\end{vmatrix} &= 0,~\text{if}~ n = 2m+1 ~(m \geq 1);\\\\
\begin{vmatrix}
    A_3 & \cdots & A_{m+1}\\
    \vdots & \ddots & \vdots \\
     A_{m+1} & \cdots & A_{2m-1}
\end{vmatrix} &= 0,~\text{if}~ n = 2m ~(m \geq 2).
\end{align*}
\end{theorem}

For the purposes of this paper, we require only the cases $n=3$ and $n=4$. Accordingly, Cayley's criteria reduce to
\[
A_2=0
\]
for triangles and
\[
A_3=0
\]
for quadrilaterals.

Write
\[
\det(\lambda\mathcal C+\mathcal D)
=
C_0+C_1\lambda+C_2\lambda^2+C_3\lambda^3,
\]
where
\begin{align*}
    C_0=\det (\mathcal D),\qquad
    C_1=\operatorname{Tr}(\mathcal C\, \operatorname{adj}(\mathcal D)),\qquad
    C_2 =\operatorname{Tr}(\operatorname{adj(\mathcal C)\,\mathcal D}),\qquad 
    C_3=\det (\mathcal C).
\end{align*}
On the other hand,
\begin{align*}
    \det(\lambda\mathcal{C}+\mathcal{D})&=(A_0+A_1\lambda+\cdots +A_n \lambda^n+\cdots)^2\\
    &=\sum_{n=0}^{\infty}
\left(
\sum_{k=0}^{n}A_kA_{n-k}
\right)\lambda^n. 
\end{align*}
Comparing the coefficients of various powers of $\lambda$, we obtain
\begin{align*}
    A_0&=\sqrt{C_0}\\
    A_1&=\frac{C_1}{2A_0}\\
    A_2&=\frac{C_2-A_1^2}{2A_0}\\
    A_3&=\frac{C_3-2A_1A_2}{2A_0}
\end{align*}
Hence Cayley's criteria become
\begin{equation}\label{eq:3ponscond}
4C_0C_2-C_1^2=0
\end{equation}
for triangles, and
\begin{equation}\label{eq:4ponscond}
8C_0^2C_3-(4C_0C_2-C_1^2)C_1=0
\end{equation}
for quadrilaterals.

We now specialize to the case where $\mathcal C$ is a circle and $\mathcal D$ is a central conic:
\begin{align}
\mathcal{C}&: (x-x_O)^2+(y-y_O)^2=R^2, \label{eq:mathcalC}\\
\mathcal{D}&: \frac{x^2}{a^2}+\varepsilon\frac{y^2}{b^2}=1 \label{eq:mathcalD}
\end{align}
where $a \geq b>0$ and 
\[
\varepsilon=
\begin{cases}
1, & \text{if }\mathcal D\text{ is an ellipse},\\
-1, & \text{if }\mathcal D\text{ is a hyperbola}.
\end{cases}
\]
Substituting the corresponding matrices into \eqref{eq:3ponscond} and \eqref{eq:4ponscond} yields explicit conditions for $3$- and $4$-Poncelet pairs, which are derived in the following subsections.

\subsection{Generalized Chapple–Euler and Cayley–Fuss relations}
Let $\mathcal C$ and $\mathcal D$ also denote the following $3\times3$ symmetric matrices representing the corresponding quadratic forms:
\begin{align*}
\mathcal{C} &=\begin{pmatrix}
1 & 0 & -x_O\\
0 & 1 & -y_O\\
-x_O & -y_O & x_O^2+y_O^2-R^2
\end{pmatrix},\\
\mathcal{D} &= \begin{pmatrix}
\frac{1}{a^2} & 0 & 0\\
0 & \frac{\varepsilon}{b^2} & 0\\
0 & 0 & -1
\end{pmatrix}.
\end{align*}
For ellipse ($\varepsilon = 1$), the corresponding coefficients in the characteristic polynomial are
\begin{align*}
    C_0&=-\frac{1}{a^2b^2}\\
    C_1&=\frac{-a^2-b^2-R^2+x_O^2+y_O^2}{a^2b^2}\\
    C_2&=-1-\frac{R^2}{a^2}-\frac{R^2}{b^2}+\frac{x_O^2}{a^2}+\frac{y_O^2}{b^2}\\
    C_3&=-R^2.
\end{align*}

Let \(F_{\pm}=\bigl(\pm\sqrt{a^2-b^2},0\bigr)\) denote the foci of the ellipse \(\mathcal D\), and let \(d_{\pm}\) be their distances from the circumcenter \(O\). Since
\[
d_{\pm}^{\,2}
=
\left(x_O\mp\sqrt{a^2-b^2}\right)^2+y_O^2,
\]
substituting these expressions into \eqref{eq:3ponscond} and \eqref{eq:4ponscond} yields the following geometric characterizations of circle--ellipse \(3\)- and \(4\)-Poncelet pairs. The corresponding characterizations for circle--hyperbola \(3\)- and \(4\)-Poncelet pairs are obtained by replacing \(b^2\) with \(-b^2\).

\begin{theorem}\label{thm:cayleychapplefuss}
A circle $\mathcal C$ of radius $R$ and central conic $\mathcal D$ of semi-axes lengths $a,b$ with $a \geq b>0$ form
\begin{itemize}
    \item a 3-Poncelet pair if and only if
\begin{equation}\label{eq:genchappleeuler}
        (R^2-d_+^2)(R^2-d_-^2)- 4\varepsilon b^2 R^2=0;
\end{equation}
    \item a 4-Poncelet pair if and only if
\begin{equation}\label{eq:cayleyfuss}
        \left(R^2-\frac{d_+^2+d_-^2}{2}\right)\left((R^2-d_+^2)(R^2-d_-^2)- 4\varepsilon b^2 R^2\right)+2a^2 (R^2-d_+^2)(R^2-d_-^2)=0,
\end{equation}
\end{itemize}
where $d_\pm$ denote the distances from the center of the circle $\mathcal C$
to the two foci of $\mathcal D$, and
\[
\varepsilon=
\begin{cases}
1, & \text{if }\mathcal D\text{ is an ellipse},\\
-1, & \text{if }\mathcal D\text{ is a hyperbola}.
\end{cases}
\]
\end{theorem}

\begin{corollary}
The circumradius $R$, inradius $r$ and the distance $d$ between the circumcenter and incenter of 
\begin{itemize}
    \item a triangle satisfy
\begin{equation}\label{eq:chappleeuler}
        (R^2 - d^2)^2=4 R^2 r^2,
\end{equation}
    \item a bicentric quadrilateral satisfy
\begin{equation}\label{eq:genfuss}
        \left(R^2-d^2\right)\left((R^2 - d^2)^2-2 r^2 (R^2 + d^2)\right)=0.
\end{equation}
\end{itemize}
\end{corollary}
\begin{proof} In the limit when the foci of the ellipse coincide i.e., $a=b$ and $d_+=d_-=d$ is the distance between the circumcenter and incenter, we may take $a=b=r$ and $\varepsilon =1$ in \eqref{eq:genchappleeuler} and \eqref{eq:cayleyfuss} which reduce to \eqref{eq:chappleeuler} and \eqref{eq:genfuss}, respectively.
\end{proof}

\begin{remark}
The formula \eqref{eq:genchappleeuler} may be referred to as the \emph{Generalized Chapple--Euler Formula}. For an alternative derivation of \eqref{eq:genchappleeuler} without using Cayley's criterion, see \cite{Dragovic-Murad2026}. The second factor in \eqref{eq:genfuss} is the classical \emph{Fuss relation} for bicentric quadrilaterals.
\end{remark}

\subsection{Focal Characterization of 3- and 4-Poncelet Pairs of a Circle and Central Conics}
\label{sec:ccc3poncp}

In this subsection we establish a characterization of central conics that form 3- and 4-Poncelet pairs with the unit circle in terms of the focal parameters $a_1,a_2$.

\begin{theorem}\label{thm:ellorhyp3}
Let $\mathbb{T}$ denote the unit circle, and let $a_1,a_2\in\mathbb{C}$ be distinct points such that $a_1,a_2\notin\mathbb{T}$ and $a_1,a_2$ are not inverse points with respect to $\mathbb{T}$. Then there exists a unique central conic $\mathcal{D}$ with foci $a_1$ and $a_2$ such that $(\mathbb{T},\mathcal{D})$ is a $3$-Poncelet pair.

Moreover,
\begin{itemize}
    \item[(a)] If $|a_1|<1$ and $|a_2|<1$, or if $|a_1|>1$ and $|a_2|>1$, then $\mathcal{D}$ is an ellipse.
    
    \item[(b)] If $|a_1|<1$ and $|a_2|>1$, then $\mathcal{D}$ is a hyperbola.
\end{itemize}

In both cases, the length of the major (transverse) axis of $\mathcal{D}$ is
\[
|1-\overline{a_1}a_2|.
\]
\end{theorem}
\begin{proof}
Let $a$ and $b$ denote the lengths of the semi-major and semi-minor axes of $\mathcal D$, respectively. Setting $R=1$, $d_+=|a_1|$, and $d_-=|a_2|$ in the generalized Chapple--Euler relation \eqref{eq:genchappleeuler}, we obtain
\begin{equation}\label{eq:3ponceletb}
(1-|a_1|^2)(1-|a_2|^2)-4\varepsilon b^2=0,
\end{equation}
where $\varepsilon=1$ corresponds to an ellipse and $\varepsilon=-1$ to a hyperbola.

We distinguish two cases.

\medskip
\noindent
\textbf{Case (a).}
Suppose that either $|a_1|<1$ and $|a_2|<1$, or $|a_1|>1$ and $|a_2|>1$. Then
\[
(1-|a_1|^2)(1-|a_2|^2)>0.
\]
Since $b^2>0$, equation \eqref{eq:3ponceletb} forces $\varepsilon=1$. Hence $\mathcal D$ is an ellipse.

\medskip
\noindent
\textbf{Case (b).}
Suppose that $|a_1|<1$ and $|a_2|>1$. Then
\[
(1-|a_1|^2)(1-|a_2|^2)<0.
\]
Equation \eqref{eq:3ponceletb} therefore implies that $\varepsilon=-1$, and consequently $\mathcal D$ is a hyperbola.

Furthermore,
\begin{equation}\label{eq:a1a2c}
2c=|a_1-a_2|,
\end{equation}
where $c$ denotes the focal distance. Since
\begin{equation}\label{eq:a2b2c2}
    a^2=\varepsilon b^2+c^2,
\end{equation}
multiplying by 4 and substituting, we obtain
\begin{align*}
    4a^2&=(1-|a_1|^2)(1-|a_2|^2)+4|a_1-a_2|^2\\
    &=|1-\overline{a_1}a_2|^2
\end{align*}
Taking square roots yields
\[
2a=|1-\overline{a_1}a_2|,
\]
which is precisely the asserted length of the major axis (or transverse axis, in the hyperbolic case).
\end{proof}

\begin{theorem}\label{thm:ellorhyp4}
Let $\mathbb{T}$ denote the unit circle, and let $a_1,a_2\in\mathbb{C}$ be distinct points such that $a_1,a_2\notin\mathbb{T}$ and $a_1,a_2$ are not inverse points with respect to $\mathbb{T}$. Then there exists a unique central conic $\mathcal{D}$ with foci $a_1$ and $a_2$ such that $(\mathbb{T},\mathcal{D})$ is a $4$-Poncelet pair.

Moreover,
\begin{itemize}
    \item[(a)] If $|a_1|<1$ and $|a_2|<1$, then $\mathcal{D}$ is an ellipse.

    \item[(b)] If $|a_1|>1$ and $|a_2|>1$, then
    \begin{itemize}
        \item if $\operatorname{Re}(1-\overline{a_1}a_2)<0$, then $\mathcal{D}$ is an ellipse;
        \item if $\operatorname{Re}(1-\overline{a_1}a_2)>0$, then $\mathcal{D}$ is a hyperbola.
    \end{itemize}

    \item[(c)] If $|a_1|<1$ and $|a_2|>1$, then
    \begin{itemize}
        \item if
        \[
        \operatorname{Re}(1-\overline{a_1}a_2)\bigl(1-|a_1|^2|a_2|^2\bigr)<0,
        \]
        then $\mathcal{D}$ is an ellipse;

        \item if
        \[
        \operatorname{Re}(1-\overline{a_1}a_2)\bigl(1-|a_1|^2|a_2|^2\bigr)>0,
        \]
        then $\mathcal{D}$ is a hyperbola.
    \end{itemize}
\end{itemize}

In each case, the length of the major (transverse) axis of $\mathcal{D}$ is
\[
|1-\overline{a_1}a_2|
\sqrt{\frac{2-|a_1|^2-|a_2|^2}{1-|a_1|^2|a_2|^2}}.
\]
\end{theorem}
\begin{proof}
Let $a$ and $b$ denote the lengths of the semi-major and semi-minor axes of $\mathcal D$, respectively. Setting $R=1$, $d_+=|a_1|$, and $d_-=|a_2|$ in the Cayley--Fuss relation \eqref{eq:cayleyfuss}, we obtain
\begin{equation}\label{eq:a1a2beta}
\left(1-\frac{|a_1|^2+|a_2|^2}{2}\right)
\left((1-|a_1|^2)(1-|a_2|^2)-4\varepsilon b^2\right)
+2a^2(1-|a_1|^2)(1-|a_2|^2)=0,
\end{equation}
where $\varepsilon=1$ corresponds to an ellipse and $\varepsilon=-1$ to a hyperbola.

Substituting \eqref{eq:a1a2c} into \eqref{eq:a1a2beta}, we obtain
\begin{align}
4\varepsilon b^2
&=(1-|a_1|^2)(1-|a_2|^2)
\left(1+\frac{|1-\overline{a_1}a_2|^2}
{1-|a_1|^2|a_2|^2}\right)\label{eq:a1a2ell}\\
&=(1-|a_1|^2)(1-|a_2|^2)
\frac{2\operatorname{Re}(1-\overline{a_1}a_2)}
{1-|a_1|^2|a_2|^2}.\label{eq:a1a2}
\end{align}

The sign of the right-hand side of \eqref{eq:a1a2} determines the sign of $\varepsilon$, exactly as in the proof of Theorem~\ref{thm:ellorhyp3}. Whenever $|a_1|<1$ and $|a_2|<1$, the right-hand side of \eqref{eq:a1a2ell} is positive in this case, and hence $\varepsilon=1$. When $|a_1|>1$ and $|a_2|>1$, the factor
\[
\frac{(1-|a_1|^2)(1-|a_2|^2)}{1-|a_1|^2|a_2|^2}
\]
is negative, so the sign of \eqref{eq:a1a2} is determined by
\[
\operatorname{Re}(1-\overline{a_1}a_2).
\]
Finally, if $|a_1|<1$ and $|a_2|>1$, then
\[
(1-|a_1|^2)(1-|a_2|^2)<0,
\]
and therefore the sign of \eqref{eq:a1a2} is determined by
\[
\frac{\operatorname{Re}(1-\overline{a_1}a_2)}
{1-|a_1|^2|a_2|^2},
\]
equivalently, by
\[
\operatorname{Re}\,(1-\overline{a_1}a_2)(1-|a_1|^2|a_2|^2),
\]
yielding precisely the classification stated in parts (a)--(c).

Finally, using \eqref{eq:a1a2c}, the identity \eqref{eq:a2b2c2} and equations \eqref{eq:a1a2ell}--\eqref{eq:a1a2}, we obtain
\[
4a^2=
|1-\overline{a_1}a_2|^2
\frac{2-|a_1|^2-|a_2|^2}
{1-|a_1|^2|a_2|^2}.
\]
Taking square roots gives the asserted length of the major axis (or transverse axis, in the hyperbolic case).
\end{proof}

\begin{remark}
The major-axis length formulas in Theorem \ref{thm:ellorhyp3} and \ref{thm:ellorhyp4} were previously obtained in \cite{daepp2002ellipses} and \cite{Fujimura2013}, respectively, in the setting of Blaschke products, where both foci lie in the unit disk. The present proofs derive the same formulas directly via Cayley's criteria and therefore apply uniformly to all central conics forming a 3- and 4-Poncelet pair with the unit circle.
\end{remark}

Theorems \ref{thm:ellorhyp3}--\ref{thm:ellorhyp4} show that the type of the central conic forming the 3- and 4-Poncelet pair with the unit circle is completely characterized by the location of its foci with respect to the unit circle. In the next section we use Siebeck--Marden theorem to show that the focal parameters and a unimodular parameter also determine the vertices of every polygon circumscribed about the central conic through explicit symmetric parametrization.

\section{The Siebeck--Marden Theorem and Symmetric Parametrizations of Cyclic Polygons
}\label{sec:siebeckmarden}
\noindent
The purpose of this section is to establish a direct connection between Siebeck--Marden theorem and the classical Blaschke-product parametrization of $p$-gons inscribed in $\mathbb T$ and circumscribed about an algebraic curve of class $p-1$. This perspective reveals that the Blaschke-product parametrization can be recovered naturally from the Siebeck–Marden theorem.

\subsection{The Siebeck--Marden theorems}
\begin{theorem}[Siebeck 1864, Marden 1945 \cite{Marden1945} I]\label{thm:siebeckmarden3}
The zeros of the partial fraction
\begin{equation*}
    F(z)=\frac{m_1}{z-z_1}+\frac{m_2}{z-z_2}+\frac{m_3}{z-z_3}, \qquad m_1m_2m_3 \ne 0
\end{equation*}
where $z_1$, $z_2$, $z_3$ are three distinct noncollinear points  lie at the foci of the conic which touches the line segments $(z_2,z_3)$, $(z_3,z_1)$ and $(z_1,z_2)$ in the points $\xi_1$, $\xi_2$ and $\xi_3$ that divide these segments in the ratio $m_2:m_3$, $m_3:m_1$ and $m_1:m_2$, respectively. If $n = m_1+m_2+m_3 \neq 0$, this conic is an ellipse or hyperbola according as $nm_1m_2m_3>0$ or $< 0$. If $n = 0$, the conic is a parabola whose axis is parallel to the line joining the origin to the point $v = m_1z_1+m_2z_2+m_3z_3$. 
\end{theorem}

\begin{theorem}[Siebeck 1864, Marden 1945 \cite{Marden1945} II]\label{thm:siebeckmardenp}
The zeros of the partial fraction
\begin{equation*}
    F(z):=\sum_{j=1}^{p}\frac{m_j}{z-z_j}, \qquad \prod_{j=1}^p m_j \ne 0
\end{equation*}
lie at the foci of the curve $C(z_1,\dots,z_p;m_1,\dots ,m_p)$ of class $p-1$ which touches each of the $p(p-1)/2$ line segments $(z_j,z_k)$ in a point dividing it in the ratio $m_j:m_k$. 
\end{theorem}

\subsection{The triangular case}
The following theorem is a special case of one of the main results of this paper.

\begin{theorem}[Parametrization of Poncelet Triangles]\label{thm:symmpara}
Let $z_1,z_2,z_3\in\mathbb T$ be the vertices of a triangle circumscribed about a central conic with foci $a_1, a_2 \in\mathbb C$. Then there exists $\lambda\in\mathbb T$ such that
\begin{align}
z_1+z_2+z_3
&=
a_1+a_2+\overline{a_1}\,\overline{a_2}\lambda,
\label{eq:s1}
\\
z_1z_2+z_2z_3+z_3z_1
&=
a_1a_2+
(\overline{a_1}+\overline{a_2})\lambda,
\label{eq:s2}
\\
z_1z_2z_3
&=
\lambda.
\label{eq:s3}
\end{align}
\end{theorem}

\begin{proof}
Let $\mathcal D$ denote the conic inscribed in the triangle $\triangle z_1z_2z_3$. Suppose that the sidelines of the triangle  touches $\mathcal D$ in the ratio determined by $m_1,m_2,m_3\in \mathbb R$. 

From the uniqueness of the inconic $\mathcal{D}$ (see \cite[Corollary 2.1]{Dragovic-Murad2026}) and Theorem \ref{thm:siebeckmarden3} it follows that the zeros of $F(z)$ are precisely the foci $a_1,a_2$.

Writing $F(z)=0$ in polynomial form yields
\begin{equation}\label{eq:compquad}
    nz^2-pz+q=0,
\end{equation}
where
\begin{align*}
n&=m_1+m_2+m_3\\
p&=m_1(z_2+z_3)+m_2(z_3+z_1)+m_3(z_1+z_2)\\
q&=m_1z_2z_3+m_2z_3z_1+m_3z_1z_2.
\end{align*}

Since the roots of \eqref{eq:compquad} are $a_1$ and $a_2$, Vieta's formulas give
\begin{align}
a_1+a_2&=\frac{p}{n},
\label{eq:vieta1}
\\
a_1a_2&=\frac{q}{n}.
\label{eq:vieta2}
\end{align}
Define
\[
\lambda=z_1z_2z_3.
\]
The assumption $z_1,z_2,z_3\in\mathbb T$ implies that $\lambda\in \mathbb T$.

Using $|z_k|=1$, we obtain
\[
\overline q\,\lambda
=
m_1z_1+m_2z_2+m_3z_3.
\]
This gives
\begin{equation}\label{eq:pqlambda}
p+\overline q\,\lambda = n(z_1+z_2+z_3).
\end{equation}
Note that, by Theorem \ref{thm:siebeckmarden3}, $n\neq 0$. So, dividing \eqref{eq:pqlambda} by $n$ and using
\eqref{eq:vieta1}--\eqref{eq:vieta2},
we obtain \eqref{eq:s1}.

Next, using the relation
\[
z_1z_2+z_2z_3+z_3z_1
=
\lambda
(\overline z_1+\overline z_2+\overline z_3),
\]
and taking conjugates in \eqref{eq:s1} and using
$\lambda \overline{\lambda}=|\lambda|^2=1$ we obtain \eqref{eq:s2}, while \eqref{eq:s3} follows directly from
the definition of $\lambda$.
\end{proof}

An important consequence of Theorem \ref{thm:symmpara} is that the
classical degree-3 Blaschke-product can be recovered directly from Siebeck--Marden theorem.

\begin{corollary}
\label{cor:mardenblaschke}
With the notation of Theorem \ref{thm:symmpara}, the vertices
$z_1,z_2,z_3$ satisfy
\begin{equation}\label{eq:classblaschke}
    z\frac{z-a_1}{1-\overline{a_1} z}
 \frac{z-a_2}{1-\overline{a_2} z}
=
\lambda.
\end{equation}
\end{corollary}

\begin{proof}
By \eqref{eq:s1}--\eqref{eq:s3}, the complex numbers $z_1,z_2,z_3$ are the roots of the cubic polynomial equation
\[
z^3
-
(a_1+a_2+\overline{a_1}\,\overline{a_2}\lambda)z^2
+
(a_1a_2+(\overline{a_1}+\overline{a_2})\lambda)z
-\lambda=0.
\]
A direct expansion shows that this polynomial is equivalent to
\[
z(z-a_1)(z-a_2)
=
\lambda
(1-\overline{a_1} z)(1-\overline{a_2} z).
\]
Dividing by
$(1-\overline{a_1} z)(1-\overline{a_2} z)$
gives the desired relation \eqref{eq:classblaschke}.
\end{proof}

Corollary \ref{cor:mardenblaschke} shows that the classical degree-3 Blaschke equation is a direct consequence of Siebeck--Marden theorem. Thus the Blaschke-product formalism arises naturally from the geometry of 3-Poncelet pairs rather than as a separate assumption.

\begin{remark}
For $\lambda\in\mathbb T$, if the solutions of \eqref{eq:classblaschke} are distinct and all lie on $\mathbb T$, then the triangle with vertices $z_1,z_2,z_3$ is circumscribed about the unique central conic whose foci are $a_1$ and $a_2$ and whose major (transverse) axis has length
$|1-\overline{a_1}a_2|$. The type of the conic is determined by Theorem \ref{thm:ellorhyp3}.

Daepp, Gorkin, and Mortini \cite{daepp2002ellipses} proved that when $|a_1|,|a_2|<1$, these hypotheses are satisfied for every $\lambda\in\mathbb T$. Thus, in the classical Blaschke setting, every
$\lambda\in\mathbb T$ determines a circumscribed triangle. \,Outside this setting, however, the equation \eqref{eq:classblaschke} may admit solutions off the unit circle, so the above \,correspondence holds only for those values of $\lambda \in \mathbb T$ for which all
three solutions are distinct and unimodular.
\end{remark}

\subsection{Parametrizations of cyclic polygons}
The preceding subsections show that, regardless of whether the associated Poncelet conic is an ellipse or a hyperbola, Siebeck--Marden theorem leads to the
same symmetric parametrization of the vertices of a circumscribed triangle. Motivated by this observation, we prove an analogous parametrization for polygons. As a consequence, generalized Möbius-product equations emerge naturally from the
same residue structure, placing the classical degree-3 Blaschke-product parametrization into a broader framework.

Throughout this section, we write
\[
e_r(x_1,\dots,x_n)
=
\sum_{1\le i_1<\cdots<i_r\le n}
x_{i_1}\cdots x_{i_r},
\qquad
0 < r\le n,
\]
for the $r$th elementary symmetric polynomial in the variables
$x_1,\dots,x_n$.

For convenience, we adopt the conventions
\[
e_0(x_1,\dots,x_n)=1,
\qquad
e_r(x_1,\dots,x_n)=0,
\quad r>n.
\]
Thus
\[
\prod_{j=1}^{n}(x-x_j)
=
\sum_{r=0}^{n}
(-1)^r
e_r(x_1,\dots,x_n)
x^{\,n-r}.
\]
These conventions allow many of the identities appearing below to be written in a uniform form without treating boundary cases separately.

\begin{lemma}\label{lem:coefficientidentity}
Let
\[
P(z)
:=
\sum_{j=1}^{n+1}
m_j
\prod_{\substack{k=1\\k\neq j}}^{n+1}(z-z_k)
=
c_0z^n-c_1z^{n-1}+\cdots+(-1)^nc_n
\]
where $m_1,\dots, m_{n+1} \in \mathbb R$. Let
\[
\lambda=(-1)^n\prod_{j=1}^{n+1}z_j.
\]
If $z_1,\dots,z_{n+1}\in\mathbb T$, then, for every \(r=1,\dots,n\),
\[
c_r+(-1)^n\,\overline{c_{n+1-r}}\,\lambda
=
c_0\,e_r(z_1,\dots,z_{n+1}).
\]
\end{lemma}

\begin{proof}
Since \(z_j\in\mathbb T\),
\[
\overline{z_j}=\frac1{z_j},
\qquad
j=1,\dots,n+1.
\]
Hence
\begin{align*}
(-1)^n\,\overline{c_{n+1-r}}\,\lambda
&=
(-1)^n
\sum_{j=1}^{n+1}
m_j
\sum_{\substack{i_1<\cdots<i_{\,n+1-r}\\ i_\nu\neq j}}
\overline{z_{i_1}\cdots z_{i_{\,n+1-r}}}\,
\lambda\\
&=
\sum_{j=1}^{n+1}
m_j
\sum_{\substack{i_1<\cdots<i_{\,n+1-r}\\ i_\nu\neq j}}
\frac{\prod_{l=1}^{n+1}z_l}
     {z_{i_1}\cdots z_{i_{\,n+1-r}}}.
\end{align*}

For each fixed \(j\), the complement of
\(\{i_1,\dots,i_{n+1-r}\}\)
inside
\(\{1,\dots,n+1\}\setminus\{j\}\)
consists of exactly \(r\) indices
\(k_1,\dots,k_r\). Therefore
\[
\frac{\prod_{l=1}^{n+1}z_l}
     {z_{i_1}\cdots z_{i_{\,n+1-r}}}
=
z_jz_{k_1}\cdots z_{k_r},
\]
and so
\[
(-1)^n\,\overline{c_{n+1-r}}\,\lambda
=
\sum_{j=1}^{n+1}
m_jz_j
\sum_{\substack{k_1<\cdots<k_r\\k_\nu\neq j}}
z_{k_1}\cdots z_{k_r}.
\]
On the other hand, every monomial of
\(e_r(z_1,\dots,z_{n+1})\)
either contains \(z_j\) or it does not. Thus
\[
e_r(z_1,\dots,z_{n+1})
=
\sum_{\substack{k_1<\cdots<k_r\\k_\nu\neq j}}
z_{k_1}\cdots z_{k_r}
+
z_j
\sum_{\substack{k_1<\cdots<k_{r-1}\\k_\nu\neq j}}
z_{k_1}\cdots z_{k_{r-1}}.
\]
Multiplying by \(m_j\) and summing over \(j\) yields
\[
c_0\,e_r(z_1,\dots,z_{n+1})
=
c_r
+
(-1)^n\,\overline{c_{n+1-r}}\,\lambda,
\]
as claimed.
\end{proof}

The following theorem extends Theorem \ref{thm:symmpara} from triangles to polygons.
\begin{theorem}\label{thm:general-symmetric}
Let $n$ be a positive integer, and let $m_j\in\mathbb R,\, j=1,\dots,n+1,$ satisfy
\[
\prod_{j=1}^{n+1}m_j\neq0,
\qquad
\sum_{j=1}^{n+1}m_j\neq0.
\]
Let $z_1,\dots,z_{n+1}\in\mathbb T$ be distinct, and let $a_1,\dots,a_n$ denote the zeros of the partial fraction
\[
F(z):=\sum_{j=1}^{n+1}\frac{m_j}{z-z_j}.
\]
Then there exists a unimodular constant $\lambda\in\mathbb T$ such that, for every $r=1,\dots,n+1$,
\begin{equation}\label{eq:ezea}
e_r(z_1,\dots,z_{n+1})
=
e_r(a_1,\dots,a_n)
+
(-1)^n
e_{n+1-r}(\overline{a_1},\dots,\overline{a_n})\,\lambda.
\end{equation}
\end{theorem}

\begin{proof}
Since
\[
\sum_{j=1}^{n+1}m_j\neq0,
\]
the numerator of $F$ has degree $n$. Accordingly, we write
\[
F(z)
=
\frac{P(z)}
{(z-z_1)\cdots(z-z_{n+1})},
\]
where
\[
P(z)
=
c_0z^n-c_1z^{n-1}+\cdots+(-1)^nc_n.
\]
Since the zeros of $P$ are precisely $a_1,\dots,a_n$, Vieta's formulas
yield
\[
e_r(a_1,\dots,a_n)=\frac{c_r}{c_0},
\qquad
r=1,\dots,n,
\]
where
\[
c_0=\sum_{j=1}^{n+1}m_j.
\]
Now define
\[
\lambda:=(-1)^n\prod_{j=1}^{n+1}z_j.
\]
Since $|z_j|=1$ for every $j$, it follows that
\[
|\lambda|=1.
\]
By Lemma~\ref{lem:coefficientidentity},
\[
c_r+(-1)^n\,\overline{c_{n+1-r}}\,\lambda
=
c_0\,e_r(z_1,\dots,z_{n+1}),
\qquad
r=1,\dots,n.
\]
Dividing by $c_0$ and using
\[
\frac{\overline{c_{n+1-r}}}{c_0}
=
\overline{\left(\frac{c_{n+1-r}}{c_0}\right)}
=
\overline{e_{n+1-r}(a_1,\dots,a_n)}
=
e_{n+1-r}(\overline{a_1},\dots,\overline{a_n}),
\]
we obtain
\[
e_r(z_1,\dots,z_{n+1})
=
e_r(a_1,\dots,a_n)
+
(-1)^n
e_{n+1-r}(\overline{a_1},\dots,\overline{a_n})\,\lambda,
\]
for every $r=1,\dots,n$.

Finally, since
\[
e_{n+1}(a_1,\dots,a_n)=0
\]
by convention, while
\[
e_0(\overline{a_1},\dots,\overline{a_n})=1,
\]
the identity also holds for $r=n+1$, namely,
\[
e_{n+1}(z_1,\dots,z_{n+1})
=
(-1)^n\lambda.
\]
This completes the proof.
\end{proof}

\begin{corollary}
With the notation of Theorem \ref{thm:general-symmetric}, the points $z_1,\dots,z_{n+1}$ satisfy the M\"obius-product equation
\begin{equation}\label{eq:genblasch}
M(z):= z\prod_{j=1}^{n}
\frac{z-a_j}
{1-\overline{a_j}z}
=\lambda.
\end{equation}
\end{corollary}

\begin{proof}
Since $z_1,\dots,z_{n+1}$ are the zeros of
\[
\prod_{j=1}^{n+1}(z-z_j),
\]
for every $j=1, \dots n+1$, $z_j$ satisfies  
\[
\sum_{k=0}^{n+1}
(-1)^k
e_k(z_1,\dots,z_{n+1})
z^{\,n+1-k}
=0.
\]
As
\[
e_{n+1}(z_1,\dots,z_{n+1})
=(-1)^n\lambda,
\]
it follows that
\[
\sum_{k=0}^{n}
(-1)^k
e_k(z_1,\dots,z_{n+1})
z^{\,n+1-k}
=\lambda.
\]
Equivalently,
\[
z
\sum_{k=0}^{n}
(-1)^k
e_k(z_1,\dots,z_{n+1})
z^{\,n-k}
=\lambda.
\]
Substituting \eqref{eq:ezea} yields
\[
z
\sum_{k=0}^{n}
(-1)^k
\Bigl(
e_k(a_1,\dots,a_n)
+
(-1)^n
e_{n+1-k}(\overline{a_1},\dots,\overline{a_n})
\lambda
\Bigr)
z^{\,n-k}
=
\lambda.
\]
Hence
\begin{align*}
z
\sum_{k=0}^{n}
(-1)^k
e_k(a_1,\dots,a_n)
z^{\,n-k}
&=
\lambda
\left(
1+
\sum_{k=1}^{n}
(-1)^{\,n+1-k}
e_{n+1-k}(\overline{a_1},\dots,\overline{a_n})
z^{\,n+1-k}
\right)\\
&=
\lambda
\sum_{k=0}^{n}
(-1)^k
e_k(\overline{a_1},\dots,\overline{a_n})
z^k.
\end{align*}

Finally, using the identities
\[
\sum_{k=0}^{n}
(-1)^k
e_k(a_1,\dots,a_n)
z^{\,n-k}
=
\prod_{j=1}^{n}(z-a_j),
\]
and
\[
\sum_{k=0}^{n}
(-1)^k
e_k(\overline{a_1},\dots,\overline{a_n})
z^k
=
\prod_{j=1}^{n}(1-\overline{a_j}z),
\]
we obtain
\[
z
\prod_{j=1}^{n}(z-a_j)
=
\lambda
\prod_{j=1}^{n}(1-\overline{a_j}z).
\]
Dividing both sides by
\[
\prod_{j=1}^{n}(1-\overline{a_j}z)
\]
gives \eqref{eq:genblasch} as claimed.
\end{proof}

\section{Geometric Applications}\label{sec:geomapp}

In this section we illustrate the effectiveness of the symmetric parametrization by deriving several geometric properties of polygons inscribed in the unit circle and circumscribed about Siebeck--Marden curve.

\subsection{Sum of Squared Lengths}
\begin{theorem}\label{thm:sumsq}
Let $\mathcal P$ be a family of $p$-gons inscribed in $\mathbb T$ and circumscribed about a Siebeck--Marden curve of class $p-1$. Let $P \in \mathcal P$. Then the sum of the squares of all sides and diagonals of the polygon $P$ remains invariant throughout $\mathcal P$ if and only if the circumcenter of $P$ coincides either with a focus of $\mathcal C_{p-1}$ or with the centroid of the foci.
\end{theorem}
\begin{proof}
Let $z_1,z_2,\dots,z_p$ denote the vertices of $P$. Since $|z_j|=1$, the generalized Apollonius' identity (see, for example, \cite[Corollary 4.5]{Murad2026b})
\[
\sum_{1\le j<k\le p}|z_j-z_k|^2
=
p\sum_{j=1}^p|z_j|^2-\left|\sum_{j=1}^pz_j\right|^2
\]
reduces to
\[
\sum_{1\le j<k\le p}|z_j-z_k|^2
=
p^2-\left|z_1+z_2+\cdots+z_p\right|^2.
\]
Hence the sum of the squares of all sides and diagonals remains invariant throughout the family $\mathcal P$ if and only if
\[
\left|z_1+z_2+\cdots+z_p\right|
\]
is independent of the parameter $\lambda$.

By \eqref{eq:ezea},
\begin{equation}\label{eq:ezear1}
z_1+\cdots+z_p
=
a_1+\cdots+a_{p-1}
+
(-1)^{p-1}\overline{a_1a_2\cdots a_{p-1}}\,\lambda,
\end{equation}
where $\lambda\in\mathbb T$. Consequently,
\[
\left|z_1+\cdots+z_p\right|
=
\left|
\Phi+\Psi\lambda
\right|,
\]
where
\[
\Phi=a_1+\cdots+a_{p-1},
\qquad
\Psi=(-1)^{p-1}\overline{a_1a_2\cdots a_{p-1}}.
\]
Since
\[
|\Phi+\Psi\lambda|^2
=
|\Phi|^2+|\Psi|^2
+
2\operatorname{Re}(\overline{\Phi}\Psi\lambda),
\]
this quantity is independent of $\lambda\in\mathbb T$ if and only if
\[
\overline{\Phi}\Psi=0,
\]
or equivalently,
\[
\Phi=0
\quad\text{or}\quad
\Psi=0.
\]
Thus,
\[
a_1+\cdots+a_{p-1}=0
\quad\text{or}\quad
a_1a_2\cdots a_{p-1}=0.
\]

The first condition is equivalent to the centroid of the foci coinciding with the circumcenter of $P$, while the second is equivalent to one of the foci coinciding with the circumcenter.

Finally, if $a_1+\cdots+a_{p-1}=0$, then
\[
|z_1+\cdots+z_p|
=
|a_1a_2\cdots a_{p-1}|,
\]
and therefore
\[
\sum_{1\le j<k\le p}|z_j-z_k|^2
=
p^2-|a_1a_2\cdots a_{p-1}|^2.
\]
On the other hand, if $a_1a_2\cdots a_{p-1}=0$, then
\[
|z_1+\cdots+z_p|
=
|a_1+\cdots+a_{p-1}|,
\]
and hence
\[
\sum_{1\le j<k\le p}|z_j-z_k|^2
=
p^2-|a_1+\cdots+a_{p-1}|^2.
\]
This completes the proof.
\end{proof}

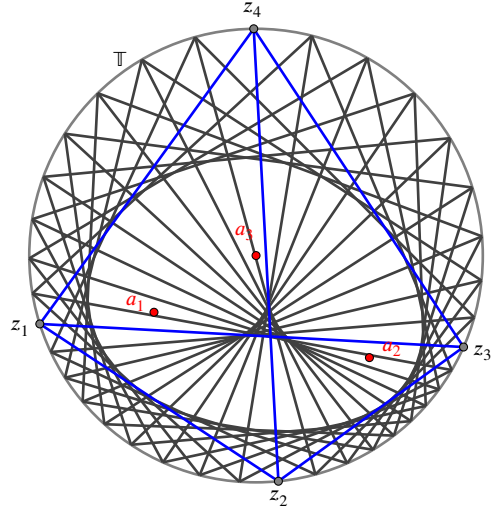
\begin{figure}
    \centering
\begin{tikzpicture}[scale=3]
\clip(-1.2,-1.2) rectangle (1.2,1.2);
\draw [line width=1.pt,gray] (0.,0.) circle (1.cm);
\draw [line width=1.pt,uququq] (-0.9050733215569846,-0.42525554976485086)-- (0.2618261195677495,-0.9651150621102621);
\draw [line width=1.pt,uququq] (0.2618261195677495,-0.9651150621102621)-- (0.9620649697444571,-0.27282044276518896);
\draw [line width=1.pt,uququq] (0.9620649697444571,-0.27282044276518896)-- (-0.268817767755222,0.9631910546403016);
\draw [line width=1.pt,uququq] (-0.268817767755222,0.9631910546403016)-- (-0.9050733215569846,-0.42525554976485086);
\draw [line width=1.pt,uququq] (-0.9050733215569846,-0.42525554976485086)-- (0.9620649697444571,-0.27282044276518896);
\draw [line width=1.pt,uququq] (0.2618261195677495,-0.9651150621102621)-- (-0.268817767755222,0.9631910546403016);
\draw [line width=1.pt,uququq] (-0.842529799280826,-0.5386497352861224)-- (0.4019466988145276,-0.915663066478114);
\draw [line width=1.pt,uququq] (0.4019466988145276,-0.915663066478114)-- (0.9939586801212003,-0.1097549188497723);
\draw [line width=1.pt,uququq] (0.9939586801212003,-0.1097549188497723)-- (-0.5033755796549019,0.86406772061401);
\draw [line width=1.pt,uququq] (-0.5033755796549019,0.86406772061401)-- (-0.842529799280826,-0.5386497352861224);
\draw [line width=1.pt,uququq] (-0.842529799280826,-0.5386497352861224)-- (0.9939586801212003,-0.1097549188497723);
\draw [line width=1.pt,uququq] (0.4019466988145276,-0.915663066478114)-- (-0.5033755796549019,0.86406772061401);
\draw [line width=1.pt,uququq] (-0.7646545414655078,-0.644440402377264)-- (0.5161697053480299,-0.8564863310531722);
\draw [line width=1.pt,uququq] (0.5161697053480299,-0.8564863310531722)-- (0.9964023660670208,0.08474859819512734);
\draw [line width=1.pt,uququq] (0.9964023660670208,0.08474859819512734)-- (-0.6979175299495429,0.7161781352353085);
\draw [line width=1.pt,uququq] (-0.6979175299495429,0.7161781352353085)-- (-0.7646545414655078,-0.644440402377264);
\draw [line width=1.pt,uququq] (-0.7646545414655078,-0.644440402377264)-- (0.9964023660670208,0.08474859819512734);
\draw [line width=1.pt,uququq] (0.5161697053480299,-0.8564863310531722)-- (-0.6979175299495429,0.7161781352353085);
\draw [line width=1.pt,uququq] (-0.8432566264135355,0.5375111738463318)-- (0.6079578029537228,-0.7939693380903832);
\draw [line width=1.pt,uququq] (0.6079578029537228,-0.7939693380903832)-- (0.9539198611228068,0.30006149129044);
\draw [line width=1.pt,uququq] (0.9539198611228068,0.30006149129044)-- (-0.6686210376629942,-0.7436033270463894);
\draw [line width=1.pt,uququq] (-0.6686210376629942,-0.7436033270463894)-- (-0.8432566264135355,0.5375111738463318);
\draw [line width=1.pt,uququq] (-0.8432566264135355,0.5375111738463318)-- (0.9539198611228068,0.30006149129044);
\draw [line width=1.pt,uququq] (0.6079578029537228,-0.7939693380903832)-- (-0.6686210376629942,-0.7436033270463894);
\draw [line width=1.pt,uququq] (-0.9379054430687092,0.3468910201519866)-- (0.6834059552898784,-0.7300385608132824);
\draw [line width=1.pt,uququq] (0.6834059552898784,-0.7300385608132824)-- (0.8556515291168636,0.5175523748568533);
\draw [line width=1.pt,uququq] (0.8556515291168636,0.5175523748568533)-- (-0.5511520413380326,-0.8344048341955582);
\draw [line width=1.pt,uququq] (-0.5511520413380326,-0.8344048341955582)-- (-0.9379054430687092,0.3468910201519866);
\draw [line width=1.pt,uququq] (-0.9379054430687092,0.3468910201519866)-- (0.8556515291168636,0.5175523748568533);
\draw [line width=1.pt,uququq] (0.6834059552898784,-0.7300385608132824)-- (-0.5511520413380326,-0.8344048341955582);
\draw [line width=1.pt,uququq] (-0.9870874959584398,0.16018200686249404)-- (0.699142743606195,0.714982114505532);
\draw [line width=1.pt,uququq] (0.699142743606195,0.714982114505532)-- (0.7483773625218225,-0.6632731890140587);
\draw [line width=1.pt,uququq] (0.7483773625218225,-0.6632731890140587)-- (-0.4104326101695778,-0.911890932353966);
\draw [line width=1.pt,uququq] (-0.4104326101695778,-0.911890932353966)-- (-0.9870874959584398,0.16018200686249404);
\draw [line width=1.pt,uququq] (-0.9870874959584398,0.16018200686249404)-- (0.7483773625218225,-0.6632731890140587);
\draw [line width=1.pt,uququq] (0.699142743606195,0.714982114505532)-- (-0.4104326101695778,-0.911890932353966);
\draw [line width=1.pt,uququq] (-0.9999224637402255,-0.012452570324093816)-- (0.4910982633254854,0.8711041819202611);
\draw [line width=1.pt,uququq] (0.4910982633254854,0.8711041819202611)-- (0.807378631638711,-0.5900336813887844);
\draw [line width=1.pt,uququq] (0.807378631638711,-0.5900336813887844)-- (-0.24855443122397092,-0.9686179302072234);
\draw [line width=1.pt,uququq] (-0.24855443122397092,-0.9686179302072234)-- (-0.9999224637402255,-0.012452570324093816);
\draw [line width=1.pt,uququq] (-0.9999224637402255,-0.012452570324093816)-- (0.807378631638711,-0.5900336813887844);
\draw [line width=1.pt,uququq] (0.4910982633254854,0.8711041819202611)-- (-0.24855443122397092,-0.9686179302072234);
\draw [line width=1.pt,uququq] (-0.9859395520965174,-0.1671023626754471)-- (0.24598954448247834,0.969272481815769);
\draw [line width=1.pt,uququq] (0.24598954448247834,0.969272481815769)-- (0.863203146874047,-0.5048567393100147);
\draw [line width=1.pt,uququq] (0.863203146874047,-0.5048567393100147)-- (-0.07325313926000795,-0.997313379830307);
\draw [line width=1.pt,uququq] (-0.07325313926000795,-0.997313379830307)-- (-0.9859395520965174,-0.1671023626754471);
\draw [line width=1.pt,uququq] (-0.9859395520965174,-0.1671023626754471)-- (0.863203146874047,-0.5048567393100147);
\draw [line width=1.pt,uququq] (0.24598954448247834,0.969272481815769)-- (-0.07325313926000795,-0.997313379830307);
\draw [line width=1.pt,blue] (-0.9534802997494435,-0.3014553333243771)-- (0.09862922708741505,-0.9951242513194706);
\draw [line width=1.pt,blue] (0.09862922708741505,-0.9951242513194706)-- (0.9150376101601777,-0.40336853123706956);
\draw [line width=1.pt,blue] (0.9150376101601777,-0.40336853123706956)-- (-0.01018653749814935,0.9999481158809178);
\draw [line width=1.pt,blue] (-0.01018653749814935,0.9999481158809178)-- (-0.9534802997494435,-0.3014553333243771);
\draw [line width=1.pt,blue] (-0.9534802997494435,-0.3014553333243771)-- (0.9150376101601777,-0.40336853123706956);
\draw [line width=1.pt,blue] (0.09862922708741505,-0.9951242513194706)-- (-0.01018653749814935,0.9999481158809178);
\begin{scriptsize}
\draw[color=black] (-0.5994760280431312,0.8659596463425003) node {$\mathbb T$};
\draw [fill=red] (-0.45,-0.25) circle (0.5pt);
\draw[color=red] (-0.53,-0.21) node {$a_1$};
\draw [fill=red] (0.5,-0.45) circle (0.5pt);
\draw[color=red] (0.6007397141379793,-0.41) node {$a_2$};
\draw [fill=red] (0.,0.) circle (0.5pt);
\draw[color=red] (-0.05,0.1000324950821869) node {$a_3$};
\draw [fill=gray] (-0.9534802997494435,-0.3014553333243771) circle (0.5pt);
\draw[color=black] (-1.04,-0.32) node {$z_1$};
\draw [fill=gray] (0.09862922708741505,-0.9951242513194706) circle (0.5pt);
\draw[color=black] (0.1,-1.07) node {$z_2$};
\draw [fill=gray] (0.9150376101601777,-0.40336853123706956) circle (0.5pt);
\draw[color=black] (1,-0.43) node {$z_3$};
\draw [fill=gray] (-0.01018653749814935,0.9999481158809178) circle (0.5pt);
\draw[color=black] (-0.015160469349695822,1.0831039418028983) node {$z_4$};
\end{scriptsize}
\end{tikzpicture}
    \caption{$M(z)$ with $a_1=-0.45-0.25i$, $a_2=0.5-0.45i$, $a_3=0$.}
    \label{fig:cenfoc3}
\end{figure}
\begin{figure}
    \centering
\begin{tikzpicture}[scale=3]
\clip(-1.2,-1.2) rectangle (1.2,1.2);
\draw [line width=1.pt,gray] (0.,0.) circle (1.cm);
\draw [line width=1.pt,uququq] (-0.9983582304628726,-0.05727864931227419)-- (0.14976920209379768,-0.9887209849619799);
\draw [line width=1.pt,uququq] (0.14976920209379768,-0.9887209849619799)-- (0.9484741376703331,0.3168545568088221);
\draw [line width=1.pt,uququq] (0.9484741376703331,0.3168545568088221)-- (-0.16912698198104778,0.9855942694466004);
\draw [line width=1.pt,uququq] (-0.16912698198104778,0.9855942694466004)-- (-0.9983582304628726,-0.05727864931227419);
\draw [line width=1.pt,uququq] (-0.9983582304628726,-0.05727864931227419)-- (0.9484741376703331,0.3168545568088221);
\draw [line width=1.pt,uququq] (-0.16912698198104778,0.9855942694466004)-- (0.14976920209379768,-0.9887209849619799);
\draw [line width=1.pt,uququq] (-0.9583258380111597,-0.28567741982909306)-- (0.23071289342841264,-0.9730218706719237);
\draw [line width=1.pt,uququq] (0.23071289342841264,-0.9730218706719237)-- (0.877115461237391,0.48027957239332825);
\draw [line width=1.pt,uququq] (0.877115461237391,0.48027957239332825)-- (-0.3676707273588133,0.9299560399521262);
\draw [line width=1.pt,uququq] (-0.3676707273588133,0.9299560399521262)-- (-0.9583258380111597,-0.28567741982909306);
\draw [line width=1.pt,uququq] (-0.9583258380111597,-0.28567741982909306)-- (0.877115461237391,0.48027957239332825);
\draw [line width=1.pt,uququq] (-0.3676707273588133,0.9299560399521262)-- (0.23071289342841264,-0.9730218706719237);
\draw [line width=1.pt,uququq] (-0.849190219188058,-0.5280870871696576)-- (0.3234865223851657,-0.9462327778275016);
\draw [line width=1.pt,uququq] (0.3234865223851657,-0.9462327778275016)-- (0.7963143701009575,0.6048829837007426);
\draw [line width=1.pt,uququq] (0.7963143701009575,0.6048829837007426)-- (-0.5350973035601145,0.8447904330203395);
\draw [line width=1.pt,uququq] (-0.5350973035601145,0.8447904330203395)-- (-0.849190219188058,-0.5280870871696576);
\draw [line width=1.pt,uququq] (-0.849190219188058,-0.5280870871696576)-- (0.7963143701009575,0.6048829837007426);
\draw [line width=1.pt,uququq] (-0.5350973035601145,0.8447904330203395)-- (0.3234865223851657,-0.9462327778275016);
\draw [line width=1.pt,uququq] (-0.6711888348236213,-0.741286414288099)-- (0.4401450600418545,-0.8979266819294054);
\draw [line width=1.pt,uququq] (0.4401450600418545,-0.8979266819294054)-- (0.7103897518117412,0.7038084970507617);
\draw [line width=1.pt,uququq] (0.7103897518117412,0.7038084970507617)-- (-0.6657588319200648,0.7461669905057645);
\draw [line width=1.pt,uququq] (-0.6657588319200648,0.7461669905057645)-- (-0.6711888348236213,-0.741286414288099);
\draw [line width=1.pt,uququq] (-0.6711888348236213,-0.741286414288099)-- (0.7103897518117412,0.7038084970507617);
\draw [line width=1.pt,uququq] (-0.6657588319200648,0.7461669905057645)-- (0.4401450600418545,-0.8979266819294054);
\draw [line width=1.pt,uququq] (-0.7658966648398586,0.64296368387894)-- (0.5897319209720842,-0.8075990721803545);
\draw [line width=1.pt,uququq] (0.5897319209720842,-0.8075990721803545)-- (0.6170594170441536,0.7869165621825019);
\draw [line width=1.pt,uququq] (0.6170594170441536,0.7869165621825019)-- (-0.46156087425849507,-0.8871085386545032);
\draw [line width=1.pt,uququq] (-0.46156087425849507,-0.8871085386545032)-- (-0.7658966648398586,0.64296368387894);
\draw [line width=1.pt,uququq] (-0.7658966648398586,0.64296368387894)-- (0.6170594170441536,0.7869165621825019);
\draw [line width=1.pt,uququq] (-0.46156087425849507,-0.8871085386545032)-- (0.5897319209720842,-0.8075990721803545);
\draw [line width=1.pt,uququq] (-0.844256979700401,0.535938571318742)-- (0.5103399980077353,0.8599727242380781);
\draw [line width=1.pt,uququq] (0.5103399980077353,0.8599727242380781)-- (0.7604453345312234,-0.6494019504049038);
\draw [line width=1.pt,uququq] (0.7604453345312234,-0.6494019504049038)-- (-0.27172826282548246,-0.9623740183430987);
\draw [line width=1.pt,uququq] (-0.27172826282548246,-0.9623740183430987)-- (-0.844256979700401,0.535938571318742);
\draw [line width=1.pt,uququq] (-0.844256979700401,0.535938571318742)-- (0.7604453345312234,-0.6494019504049038);
\draw [line width=1.pt,uququq] (-0.27172826282548246,-0.9623740183430987)-- (0.5103399980077353,0.8599727242380781);
\draw [line width=1.pt,uququq] (-0.9072440982065451,0.4206045010094302)-- (0.3818683842812295,0.924216715433367);
\draw [line width=1.pt,uququq] (0.3818683842812295,0.924216715433367)-- (0.9083952991415635,-0.4181124017504219);
\draw [line width=1.pt,uququq] (0.9083952991415635,-0.4181124017504219)-- (-0.1255581052592584,-0.9920862675209774);
\draw [line width=1.pt,uququq] (-0.1255581052592584,-0.9920862675209774)-- (-0.9072440982065451,0.4206045010094302);
\draw [line width=1.pt,uququq] (-0.9072440982065451,0.4206045010094302)-- (0.9083952991415635,-0.4181124017504219);
\draw [line width=1.pt,uququq] (-0.1255581052592584,-0.9920862675209774)-- (0.3818683842812295,0.924216715433367);
\draw [line width=1.pt,uququq] (-0.9573005585735974,0.2890945183753555)-- (0.2233594905888518,0.974736137610527);
\draw [line width=1.pt,uququq] (0.2233594905888518,0.974736137610527)-- (0.988972270688621,-0.14810080286410493);
\draw [line width=1.pt,uququq] (0.988972270688621,-0.14810080286410493)-- (-0.015996489773338746,-0.9998720479716057);
\draw [line width=1.pt,uququq] (-0.015996489773338746,-0.9998720479716057)-- (-0.9573005585735974,0.2890945183753555);
\draw [line width=1.pt,uququq] (-0.9573005585735974,0.2890945183753555)-- (0.988972270688621,-0.14810080286410493);
\draw [line width=1.pt,blue] (-0.990789180923808,0.13541343716311813)-- (0.0702084150536062,-0.9975323445661598);
\draw [line width=1.pt,blue] (0.0702084150536062,-0.9975323445661598)-- (0.9946191219187984,0.10359923896186936);
\draw [line width=1.pt,blue] (0.9946191219187984,0.10359923896186936)-- (0.0382116439514034,0.9992696684411727);
\draw [line width=1.pt,blue] (0.0382116439514034,0.9992696684411727)-- (-0.990789180923808,0.13541343716311813);
\draw [line width=1.pt,blue] (-0.990789180923808,0.13541343716311813)-- (0.9946191219187984,0.10359923896186936);
\draw [line width=1.pt,blue] (0.0382116439514034,0.9992696684411727)-- (0.0702084150536062,-0.9975323445661598);
\begin{scriptsize}
\draw [fill=green] (0.,0.) circle (0.5pt);
\draw[color=green] (0.03003368472355899,0.06716428477710368) node {$O$};
\draw[color=black] (-0.6312740976931173,0.8782021311371773) node {$\mathbb T$};
\draw [fill=red] (-0.5,0.3) circle (0.5pt);
\draw[color=red] (-0.45,0.37) node {$a_1$};
\draw [fill=red] (0.4,0.45) circle (0.5pt);
\draw[color=red] (0.44,0.5205136450501705) node {$a_2$};
\draw [fill=red] (0.1,-0.75) circle (0.5pt);
\draw[color=red] (0.02,-0.7) node {$a_3$};
\draw [fill=gray] (-0.990789180923808,0.13541343716311813) circle (0.5pt);
\draw[color=black] (-1.07,0.12) node {$z_1$};
\draw [fill=gray] (0.0702084150536062,-0.9975323445661598) circle (0.5pt);
\draw[color=black] (0.08,-1.08) node {$z_2$};
\draw [fill=gray] (0.9946191219187984,0.10359923896186936) circle (0.5pt);
\draw[color=black] (1.07,0.12) node {$z_3$};
\draw [fill=gray] (0.0382116439514034,0.9992696684411727) circle (0.5pt);
\draw[color=black] (0.05,1.07) node {$z_4$};
\end{scriptsize}
\end{tikzpicture}
    \caption{$M(z)$ with $a_1=-0.5+0.3i$, $a_2=0.4+0.45i$, $a_3=0.1-0.75i$.}
    \label{fig:cencen3}
\end{figure}
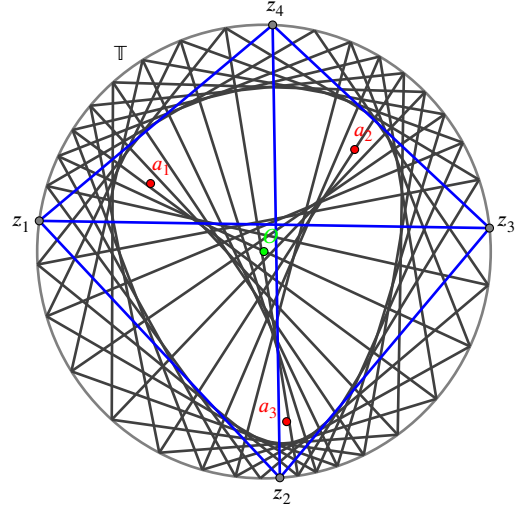
\input{Figures/fig.cencen4}
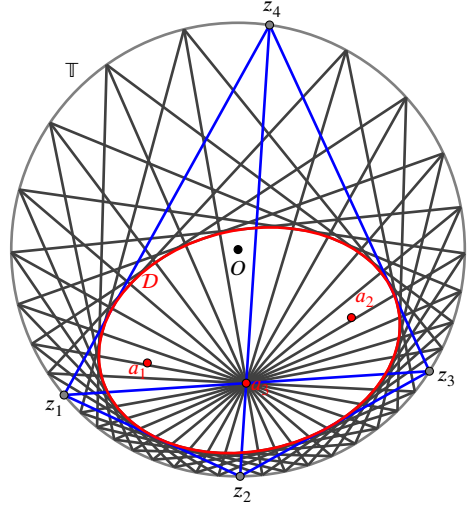
\begin{figure}
    \centering
\begin{tikzpicture}[scale=3]
\clip(-1.2,-1.2) rectangle (1.2,1.2);
\draw [line width=1.pt,gray] (0.,0.) circle (1.cm);
\draw [line width=1.pt,uququq] (-0.23816886117947017,0.9712237608113128)-- (-0.694861648531218,-0.7191434414638558);
\draw [line width=1.pt,uququq] (-0.694861648531218,-0.7191434414638558)-- (0.10720931390644893,-0.9942364723805414);
\draw [line width=1.pt,uququq] (0.10720931390644893,-0.9942364723805414)-- (0.9002449175974009,-0.43538384023760685);
\draw [line width=1.pt,uququq] (0.9002449175974009,-0.43538384023760685)-- (-0.23816886117947017,0.9712237608113128);
\draw [line width=1.pt,uququq] (-0.23816886117947017,0.9712237608113128)-- (0.10720931390644893,-0.9942364723805414);
\draw [line width=1.pt,uququq] (-0.694861648531218,-0.7191434414638558)-- (0.9002449175974009,-0.43538384023760685);
\draw [line width=1.pt,uququq] (-0.579559161560559,0.8149301677145235)-- (-0.6218312936300732,-0.7831512256661228);
\draw [line width=1.pt,uququq] (-0.6218312936300732,-0.7831512256661228)-- (0.20669445808217468,-0.9784055401509703);
\draw [line width=1.pt,uququq] (0.20669445808217468,-0.9784055401509703)-- (0.9472988157401382,-0.3203512973242523);
\draw [line width=1.pt,uququq] (0.9472988157401382,-0.3203512973242523)-- (-0.579559161560559,0.8149301677145235);
\draw [line width=1.pt,uququq] (-0.579559161560559,0.8149301677145235)-- (0.20669445808217468,-0.9784055401509703);
\draw [line width=1.pt,uququq] (-0.6218312936300732,-0.7831512256661228)-- (0.9472988157401382,-0.3203512973242523);
\draw [line width=1.pt,uququq] (-0.5456095388828103,-0.838039516419176)-- (-0.8290342518392779,0.5591978266027888);
\draw [line width=1.pt,uququq] (-0.8290342518392779,0.5591978266027888)-- (0.30874069051955766,-0.9511462484904751);
\draw [line width=1.pt,uququq] (0.30874069051955766,-0.9511462484904751)-- (0.9829496312816247,-0.18387501832308886);
\draw [line width=1.pt,uququq] (0.9829496312816247,-0.18387501832308886)-- (-0.5456095388828103,-0.838039516419176);
\draw [line width=1.pt,uququq] (-0.5456095388828103,-0.838039516419176)-- (0.30874069051955766,-0.9511462484904751);
\draw [line width=1.pt,uququq] (-0.8290342518392779,0.5591978266027888)-- (0.9829496312816247,-0.18387501832308886);
\draw [line width=1.pt,uququq] (-0.46393934880883875,-0.8858669655353625)-- (-0.9643650150397861,0.2645753536656687);
\draw [line width=1.pt,uququq] (-0.9643650150397861,0.2645753536656687)-- (0.41292738650255534,-0.9107639504702406);
\draw [line width=1.pt,uququq] (0.41292738650255534,-0.9107639504702406)-- (0.9998545140997726,-0.01705727511391688);
\draw [line width=1.pt,uququq] (0.9998545140997726,-0.01705727511391688)-- (-0.46393934880883875,-0.8858669655353625);
\draw [line width=1.pt,uququq] (-0.46393934880883875,-0.8858669655353625)-- (0.41292738650255534,-0.9107639504702406);
\draw [line width=1.pt,uququq] (-0.9643650150397861,0.2645753536656687)-- (0.9998545140997726,-0.01705727511391688);
\draw [line width=1.pt,uququq] (-0.37571240531342553,-0.9267363101193304)-- (-0.9999165068105681,-0.012922051220501355);
\draw [line width=1.pt,uququq] (-0.9999165068105681,-0.012922051220501355)-- (0.5164266898134726,-0.8563314043337995);
\draw [line width=1.pt,uququq] (0.5164266898134726,-0.8563314043337995)-- (0.9823842023587358,0.18687236006426272);
\draw [line width=1.pt,uququq] (0.9823842023587358,0.18687236006426272)-- (-0.37571240531342553,-0.9267363101193304);
\draw [line width=1.pt,uququq] (-0.37571240531342553,-0.9267363101193304)-- (0.5164266898134726,-0.8563314043337995);
\draw [line width=1.pt,uququq] (-0.9999165068105681,-0.012922051220501355)-- (0.9823842023587358,0.18687236006426272);
\draw [line width=1.pt,uququq] (-0.2817363833761367,-0.9594918500342429)-- (-0.970377642059259,-0.24159311205311818);
\draw [line width=1.pt,uququq] (-0.970377642059259,-0.24159311205311818)-- (0.6145285664662872,-0.78889456899949);
\draw [line width=1.pt,uququq] (0.6145285664662872,-0.78889456899949)-- (0.9055104529336089,0.4243239559911395);
\draw [line width=1.pt,uququq] (0.9055104529336089,0.4243239559911395)-- (-0.2817363833761367,-0.9594918500342429);
\draw [line width=1.pt,uququq] (-0.2817363833761367,-0.9594918500342429)-- (0.6145285664662872,-0.78889456899949);
\draw [line width=1.pt,uququq] (-0.970377642059259,-0.24159311205311818)-- (0.9055104529336089,0.4243239559911395);
\draw [line width=1.pt,uququq] (-0.1844318816149174,-0.9828452986324865)-- (-0.9093868108296528,-0.4159514734786627);
\draw [line width=1.pt,uququq] (-0.9093868108296528,-0.4159514734786627)-- (0.7029096172000764,-0.7112791786968371);
\draw [line width=1.pt,uququq] (0.7029096172000764,-0.7112791786968371)-- (0.7415407526302566,0.6709078268946863);
\draw [line width=1.pt,uququq] (0.7415407526302566,0.6709078268946863)-- (-0.1844318816149174,-0.9828452986324865);
\draw [line width=1.pt,uququq] (-0.1844318816149174,-0.9828452986324865)-- (0.7029096172000764,-0.7112791786968371);
\draw [line width=1.pt,uququq] (-0.9093868108296528,-0.4159514734786627)-- (0.7415407526302566,0.6709078268946863);
\draw [line width=1.pt,uququq] (-0.08629363836908077,-0.9962697465932746)-- (-0.8380436359055917,-0.5456032114257895);
\draw [line width=1.pt,uququq] (-0.8380436359055917,-0.5456032114257895)-- (0.47704042357329784,0.8788813539249819);
\draw [line width=1.pt,uququq] (0.47704042357329784,0.8788813539249819)-- (0.7797006354727715,-0.6261524726800617);
\draw [line width=1.pt,uququq] (0.7797006354727715,-0.6261524726800617)-- (-0.08629363836908077,-0.9962697465932746);
\draw [line width=1.pt,uququq] (-0.08629363836908077,-0.9962697465932746)-- (0.47704042357329784,0.8788813539249819);
\draw [line width=1.pt,uququq] (-0.8380436359055917,-0.5456032114257895)-- (0.7797006354727715,-0.6261524726800617);
\draw [line width=1.pt,blue] (0.009203395890306412,-0.999957647855191)-- (-0.7670828281667107,-0.6415480767111382);
\draw [line width=1.pt,blue] (-0.7670828281667107,-0.6415480767111382)-- (0.13911979003735347,0.9902755596398221);
\draw [line width=1.pt,blue] (0.13911979003735347,0.9902755596398221)-- (0.8439746085416304,-0.5363831281248522);
\draw [line width=1.pt,blue] (0.8439746085416304,-0.5363831281248522)-- (0.009203395890306412,-0.999957647855191);
\draw [line width=1.pt,blue] (0.009203395890306412,-0.999957647855191)-- (0.13911979003735347,0.9902755596398221);
\draw [line width=1.pt,blue] (-0.7670828281667107,-0.6415480767111382)-- (0.8439746085416304,-0.5363831281248522);
\draw [rotate around={12.528807709151506:(0.05,-0.4)},line width=1.pt,color=red] (0.05,-0.4) ellipse (0.6708572012270421cm and 0.4873903819713514cm);
\begin{scriptsize}
\draw [fill=black] (0.,0.) circle (0.5pt);
\draw[color=black] (0,-0.08) node {$O$};
\draw[color=black] (-0.731142861144204,0.7906179390299733) node {$\mathbb T$};
\draw [fill=red] (-0.4,-0.5) circle (0.5pt);
\draw[color=red] (-0.44,-0.54) node {$a_1$};
\draw [fill=red] (0.5,-0.3) circle (0.5pt);
\draw[color=red] (0.56,-0.22) node {$a_2$};
\draw [fill=red] (0.03602138043225656,-0.5891238670694864) circle (0.5pt);
\draw[color=red] (0.1,-0.61) node {$a_3$};
\draw [fill=gray] (-0.7670828281667107,-0.6415480767111382) circle (0.5pt);
\draw[color=black] (-0.81,-0.7) node {$z_1$};
\draw [fill=gray] (0.009203395890306412,-0.999957647855191) circle (0.5pt);
\draw[color=black] (0.02,-1.08) node {$z_2$};
\draw [fill=gray] (0.8439746085416304,-0.5363831281248522) circle (0.5pt);
\draw[color=black] (0.92,-0.57) node {$z_3$};
\draw [fill=gray] (0.13911979003735347,0.9902755596398221) circle (0.5pt);
\draw[color=black] (0.15,1.06) node {$z_4$};
\draw[color=red] (-0.39,-0.14) node {$\mathcal D$};
\end{scriptsize}
\end{tikzpicture}
    \caption{$M(z)$ with $a_1=-0.4-0.5i$, $a_2=0.5-0.3i$, $a_3$ is calculated by \eqref{eq:a3}.}
    \label{fig:ellipse}
\end{figure}

\begin{corollary}\label{cor:gamma}
Let $z_1,z_2, \dots,z_p \in \mathbb T$ be the vertices of a $p$-gon circumscribed about a Siebeck--Marden curve of class $p-1$ with real foci $a_1, \dots ,a_{p-1}$. Then the centroid
\[
z_G:=\frac{z_1+ \dots+z_p}{p}
\]
of the $p$-gon lies on the circle $\Gamma$ with center 
\[
\frac{a_1+ \cdots +a_{p-1}}{p}
\]
and radius
\[
\frac{1}{p}|a_1 \cdots a_{p-1}|.
\]
\end{corollary}
\begin{proof}
This follows immediately from \eqref{eq:ezea}.
\end{proof}

\begin{figure}
  \begin{subfigure}[b]{0.4\textwidth}
    \centering
\begin{tikzpicture}[scale=2.1]
\clip(-1.2,-1.4) rectangle (1.2,1.2);
\draw [line width=1.pt] (0.,0.) circle (1.cm);
\draw [rotate around={8.130102354155984:(0.05,-0.45)},line width=1.pt,color=red] (0.05,-0.45) ellipse (0.4870318280275454cm and 0.33496268674563245cm);
\draw [line width=1.pt,color=blue] (-0.29701349243427755,0.954873282332265)-- (-0.4734897603533804,-0.8807993226839462);
\draw [line width=1.pt,color=blue] (-0.4734897603533804,-0.8807993226839462)-- (0.742165894161016,-0.6702162229789574);
\draw [line width=1.pt,color=blue] (0.742165894161016,-0.6702162229789574)-- (-0.29701349243427755,0.954873282332265);
\draw [line width=1.pt,color=zzttqq] (0.1,-0.9) circle (0.32984845004941277cm);
\draw [line width=1.pt,color=qqwuqq] (0.20015045541751703,-1.2142767670058732)-- (0.4010075633955934,-0.7651132075455577);
\draw [line width=1.pt,color=qqwuqq] (0.4010075633955934,-0.7651132075455577)-- (-0.21532202560606453,-0.803190805355646);
\draw [line width=1.pt,color=qqwuqq] (-0.21532202560606453,-0.803190805355646)-- (0.20015045541751703,-1.2142767670058732);
\begin{scriptsize}
\draw [fill=wewdxt] (0.,0.) circle (0.6pt);
\draw[color=wewdxt] (0.03514880746677769,0.07945618393525569) node {$O$};
\draw[color=black] (-0.66188243730359,0.8576486010419004) node {$\mathbb T$};
\draw [fill=black] (-0.3,-0.5) circle (0.6pt);
\draw[color=black] (-0.25368948231820343,-0.40989794338640745) node {$a_1$};
\draw [fill=black] (0.4,-0.4) circle (0.6pt);
\draw[color=black] (0.42,-0.3) node {$a_2$};
\draw[color=red] (-0.23220564258213047,-0.10196290716936088) node {$\mathcal D$};
\draw [fill=wewdxt] (-0.29701349243427755,0.954873282332265) circle (0.6pt);
\draw[color=wewdxt] (-0.3,1.05) node {$A$};
\draw [fill=wewdxt] (-0.4734897603533804,-0.8807993226839462) circle (0.6pt);
\draw[color=wewdxt] (-0.54,-0.96) node {$B$};
\draw [fill=wewdxt] (0.742165894161016,-0.6702162229789574) circle (0.6pt);
\draw[color=wewdxt] (0.8,-0.75) node {$C$};
\draw [fill=blue] (-0.02833735862664194,-0.5961422633306387) circle (0.6pt);
\draw[color=blue] (0,-0.5) node {$H$};
\draw[color=zzttqq] (0.2,-0.67) node {$\Gamma$};
\draw [fill=wewdxt] (0.20015045541751703,-1.2142767670058732) circle (0.6pt);
\draw[color=wewdxt] (0.20,-1.31) node {$D$};
\draw [fill=wewdxt] (0.4010075633955934,-0.7651132075455577) circle (0.6pt);
\draw[color=wewdxt] (0.48,-0.78) node {$E$};
\draw [fill=wewdxt] (-0.21532202560606453,-0.803190805355646) circle (0.6pt);
\draw[color=wewdxt] (-0.3,-0.78) node {$F$};
\draw [fill=qqwuqq] (0.185835993207046,-0.982580779907077) circle (0.6pt);
\draw[color=qqwuqq] (0.22,-0.88) node {$H'$};
\end{scriptsize}
\end{tikzpicture}
    \caption{$\mathcal D$ is an ellipse}
    \label{fig:orthcirc(A)}
    \end{subfigure}
  \begin{subfigure}[b]{0.45\textwidth}
    \centering
\definecolor{qqwuqq}{rgb}{0.,0.39215686274509803,0.}
\begin{tikzpicture}[scale=2.]
\clip(-1.5,-2) rectangle (3,2);
\draw [line width=1.pt] (0.,0.) circle (1.cm);
\draw [samples=50,domain=-0.99:0.99,rotate around={35.83765295427828:(0.6,-0.15)},xshift=0.6cm,yshift=-0.15cm,line width=1.pt,color=red] plot ({1.0636023712341884*(1+(\x)^2)/(1-(\x)^2)},{0.3181980450992955*2*(\x)/(1-(\x)^2)});
\draw [samples=50,domain=-0.99:0.99,rotate around={35.83765295427828:(0.6,-0.15)},xshift=0.6cm,yshift=-0.15cm,line width=1.pt,color=red] plot ({1.0636023712341884*(-1-(\x)^2)/(1-(\x)^2)},{0.3181980450992955*(-2)*(\x)/(1-(\x)^2)});
\draw [line width=1.pt,color=zzttqq] (1.2,-0.3) circle (1.3509256086106296cm);
\draw [line width=1.pt,color=qqwuqq] (0.7255680024340596,0.9648771796840954)-- (-0.10223656529286185,-0.6594160931486652);
\draw [line width=1.pt,color=qqwuqq] (-0.10223656529286185,-0.6594160931486652)-- (2.5506403813511196,-0.272241032733447);
\draw [line width=1.pt,color=qqwuqq] (2.5506403813511196,-0.272241032733447)-- (0.7255680024340596,0.9648771796840954);
\draw [line width=1.pt,color=blue] (-0.944776431597904,0.32771556920160333)-- (-0.12588168256050736,-0.9920452620701012);
\draw [line width=1.pt,color=blue] (-0.12588168256050736,-0.9920452620701012)-- (0.9200135880789069,0.3918864602792184);
\draw [line width=1.pt,color=blue] (0.9200135880789069,0.3918864602792184)-- (-0.944776431597904,0.32771556920160333);
\begin{scriptsize}
\draw [fill=wewdxt] (0.,0.) circle (0.7pt);
\draw[color=wewdxt] (0.04318848061106703,0.10748315747039414) node {$O$};
\draw[color=black] (-0.5995676432836226,0.9076489443596993) node {$\mathbb T$};
\draw [fill=black] (-0.3,-0.8) circle (0.7pt);
\draw[color=black] (-0.4,-0.81) node {$a_1$};
\draw [fill=black] (1.5,0.5) circle (0.7pt);
\draw[color=black] (1.59,0.6) node {$a_2$};
\draw[color=red] (1.899310756755936,1.5700812761287144) node {$\mathcal D$};
\draw [fill=wewdxt] (-0.944776431597904,0.32771556920160333) circle (0.7pt);
\draw[color=wewdxt] (-1.04,0.37) node {$A$};
\draw [fill=wewdxt] (-0.12588168256050736,-0.9920452620701012) circle (0.7pt);
\draw[color=wewdxt] (-0.15,-1.09) node {$B$};
\draw [fill=wewdxt] (0.9200135880789069,0.3918864602792184) circle (0.7pt);
\draw[color=wewdxt] (1,0.43) node {$C$};
\draw [fill=blue] (-0.1506445260795046,-0.2724432325892794) circle (0.7pt);
\draw[color=blue] (-0.25,-0.25) node {$H$};
\draw[color=zzttqq] (1.1,1.15) node {$\Gamma$};
\draw [fill=wewdxt] (0.7255680024340596,0.9648771796840954) circle (0.7pt);
\draw[color=wewdxt] (0.75,1.09) node {$D$};
\draw [fill=wewdxt] (-0.10223656529286185,-0.6594160931486652) circle (0.7pt);
\draw[color=wewdxt] (-0.19,-0.65) node {$E$};
\draw [fill=wewdxt] (2.5506403813511196,-0.272241032733447) circle (0.7pt);
\draw[color=wewdxt] (2.64,-0.27) node {$F$};
\draw [fill=qqwuqq] (0.773971818492316,0.6332200538019839) circle (0.7pt);
\draw[color=qqwuqq] (0.8367955315422653,0.7436805454069728) node {$H'$};
\end{scriptsize}
\end{tikzpicture}
    \caption{$\mathcal D$ is a hyperbola}
    \label{fig:orthcirc(B)}
    \end{subfigure}
        \caption{$\triangle ABC$ is inscribed in $\mathbb T$ and circumscribed about $\mathcal D$. The orthocenter $H$ of $\triangle ABC$ lies on the circle $\Gamma$ with center $a_1+a_2$ and radius $|a_1a_2|$. $(\Gamma,\mathcal D)$ is also a 3-Poncelet pair. The orthocenter $H'$ of the triangle $\triangle DEF$ inscribed in $\Gamma$ and circumscribed about $\mathcal D$ lies on $\mathbb T$.}
    \label{fig:orthcirc}
\end{figure}
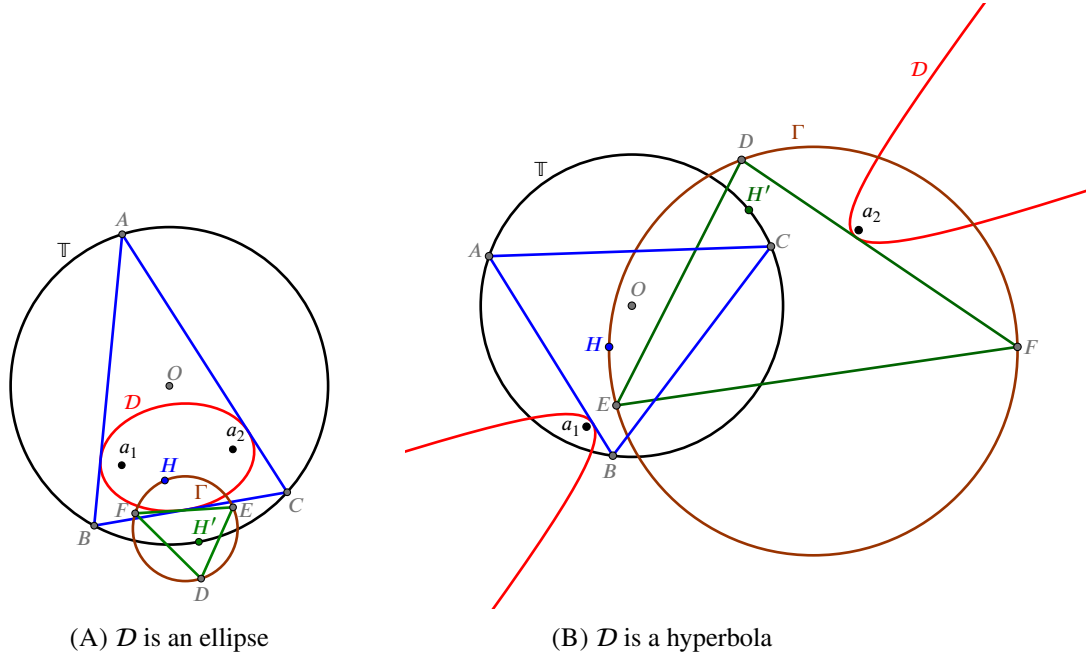

\begin{corollary}\label{cor:anticen} Let $Q$ be a cyclic quadrilateral circumscribed about a Siebeck--Marden curve, $\mathcal C_{3}$, of class 3 with three real foci. If the circumcenter of $Q$ coincides with a focus of $\mathcal C_3$, then the anticenter of $Q$ coincides with the midpoint of the remaining two foci. 
\end{corollary}
\begin{proof} Assume, without loss of generality, $a_3=0$. Then
    \[
    z_1+z_2+z_3+z_4=a_1+a_2
    \]
gives
\[
2z_T=a_1+a_2.
\]
This completes the proof.
\end{proof}

In \cite{Dragovic-Murad2026}, we proved that if a triangle $\triangle ABC$ is circumscribed about a central conic, then the sum of the squared side lengths,
\[
|AB|^{2}+|BC|^{2}+|CA|^{2},
\]
remains invariant throughout the associated Poncelet family if and only if the circumcenter of $\triangle ABC$ coincides either with the center of the conic or with one of its foci.

In the next subsections we generalize this result.

\subsection{Triangles}
In particular, if $p=3$, then the Siebeck--Marden curve of class 2 are exactly the ellipse or hyperbolas. The preceding results naturally lead to another invariant involving circles---\emph{power circles}. Power circles constitute a natural family of circles associated with a triangle. 

In this subsection, we investigate their behavior in a Poncelet family and identify the geometric configurations for which they remain invariant.
\begin{definition}
   The circles each passing through a vertex of a triangle and centered at the midpoint of the opposite side are called the \emph{power circles} of the triangle.
\end{definition}

The following theorem extends \cite[Theorem 2.1]{Celiketal2026}. Besides providing a substantially shorter proof, it establishes necessary and sufficient conditions for the invariance property. 

For the convenience of the reader, we include the result here in a slightly reformulated form.

\begin{theorem}\label{thm:powercirclearea}
Let $\mathcal P$ be a family of triangles inscribed in a circle and circumscribed about a central conic $\mathcal D$. Then the total area of the power circles of the triangles in $\mathcal P$ is invariant throughout the family if and only if the orthocenter of the triangle coincides either with the center or with a focus of $\mathcal D$.

In particular, if the circumcircle is $\mathbb T$ and $a_1,a_2$ denote the foci of $\mathcal D$, then
\[
\operatorname{Area}(\mathcal C_1)
+\operatorname{Area}(\mathcal C_2)
+\operatorname{Area}(\mathcal C_3)
=
\frac{3\pi}{4}(9-|a_1|^4)
\]
whenever the circumcenter coincides with the center of $\mathcal D$, and
\[
\operatorname{Area}(\mathcal C_1)
+\operatorname{Area}(\mathcal C_2)
+\operatorname{Area}(\mathcal C_3)
=
\frac{3\pi}{4}(9-|a_1|^2)
\]
whenever the circumcenter coincides with the focus $a_2$.
\end{theorem}

\begin{proof}
The orthocenter of the triangle $\triangle z_1z_2z_3\in\mathcal P$ is given by $z_H=z_1+z_2+z_3$. Since
\[
z_1-\frac{z_2+z_3}{2}
=
\frac32\left(z_1-\frac{z_H}{3}\right),
\]
and similarly for the other two vertices, we obtain
\[
\left|z_1-\frac{z_2+z_3}{2}\right|^2
+
\left|z_2-\frac{z_3+z_1}{2}\right|^2
+
\left|z_3-\frac{z_1+z_2}{2}\right|^2
=
\frac{27}{4}-\frac34|z_H|^2.
\]

Hence,
\[
\operatorname{Area}(\mathcal C_1)
+\operatorname{Area}(\mathcal C_2)
+\operatorname{Area}(\mathcal C_3)
=
\frac{\pi}{3}
\sum_{k=1}^{3}
\left|z_k-\frac{z_{k+1}+z_{k+2}}{2}\right|^2
=
\frac{3\pi}{4}(9-|z_H|^2),
\]
where $\mathcal C_k$ denotes the power circle through the vertex $z_k$.

Therefore, the total area of the power circles remains invariant throughout $\mathcal P$ if and only if $|z_H|$ remains invariant. The conclusion now follows immediately from Theorem \ref{thm:symmpara}. The stated values are obtained by substituting
\[
|z_H|=|a_1|^2
\]
when the circumcenter coincides with the center of $\mathcal D$, and
\[
|z_H|=|a_1|
\]
when it coincides with the focus $a_2$.
\end{proof}

\begin{corollary}
The sum of the areas of the regular convex polygons on the sides of a triangle in $\mathcal P$ satisfies the orthocenter criterion.
\end{corollary}
\begin{proof} Let $\triangle ABC \in \mathcal P$. For each vertex $X \in \{A,B,C\}$, let $P_X$ denote the regular convex $n$-gon on the side containing the remaining two vertices of $\triangle ABC$. Then
\[
\operatorname{Area}(P_A)+\operatorname{Area}(P_B)+\operatorname{Area}(P_C)=\frac{1}{4}n\left(|AB|^2+|BC|^2+|CA|^2\right)\cot\left(\frac{\pi}{n}\right)
\]
In particular, the sum of the areas of the outer (equivalently, inner) Napoleon triangles of $\triangle ABC$ satisfies the orthocenter criterion. For further background on Napoleon triangles, see, for example, \cite{Hahn2019}.
\end{proof}

\subsection{Quadrilaterals}
We now derive some properties of cyclic quadrilaterals circumscribed about a central conic.

If $p=4$, then the Siebeck--Marden curves are algebraic curves of class 3 with three real foci $a_1,a_2,a_3$. If $a_1,a_2 \in \mathbb D$, Fujimara proved that these algebraic curve reduces to an ellipse with foci $a_1,a_2$ if and only if 
\begin{equation}\label{eq:a3}
    a_3=\frac{a_1+a_2-(\overline{a_1}+\overline{a_2})a_1a_2}{1-|a_1|^2|a_2|^2}.
\end{equation}
This condition is equivalent to the corresponding Blaschke product of degree 3 being decomposable. It can be shown that this formula holds when one or both foci lie outside $\mathbb T$. In this case, $a_3$ is the intersection point of the diagonals of the quadrilateral circumscribed about the central conic.

\begin{proposition}
Let $Q$ be a cyclic quadrilateral circumscribed about a central conic $\mathcal D$. The intersection point of the diagonals of $Q$ lies inside the circumcircle of $Q$ if and only if the foci of $\mathcal D$ both lie on the same side of the circumcircle (that is, both inside or both outside). Equivalently, the diagonals intersect outside the circumcircle precisely when one focus lies inside the circumcircle and the other lies outside.
\end{proposition}

\begin{proof}
Suppose first that the foci $a_1,a_2$ satisfy
$|a_1|<1$, $|a_2|<1$. From \eqref{eq:a3}, we have
\[
a_3=\frac{a_1(1-|a_2|^2)+a_2(1-|a_1|^2)}
     {1-|a_1|^2|a_2|^2}.
\]
Hence, by the triangle inequality,
\begin{align*}
|a_3|
&\le
\frac{|a_1|(1-|a_2|^2)+|a_2|(1-|a_1|^2)}
     {1-|a_1|^2|a_2|^2} \\
&=
\frac{(|a_1|+|a_2|)(1-|a_1||a_2|)}
     {(1-|a_1|^2|a_2|^2)} \\
&=
\frac{|a_1|+|a_2|}{1+|a_1||a_2|}.
\end{align*}
Since
\[
\frac{|a_1|+|a_2|}{1+|a_1||a_2|}
<1
\iff
(1-|a_1|)(1-|a_2|)>0,
\]
it follows that $|a_3|<1$. Thus the diagonals intersect inside the circumcircle.

The proof for the case $|a_1|>1$, $|a_2|>1$ is similar. 

Now we suppose that one focus lies inside the circumcircle and the other lies outside. For definiteness of argument, assume that $|a_1|>1>|a_2|$. Then applying the reverse triangle inequality to the following
\[
a_3=\frac{a_1(1-|a_2|^2)-a_2(|a_1|^2-1)}
     {1-|a_1|^2|a_2|^2}
\]
we obtain
\begin{align*}
    |a_3|& \ge 
\left|\frac{|a_1|(1-|a_2|^2)-|a_2|(|a_1|^2-1)}
     {\left|1-|a_1|^2|a_2|^2\right|}\right| \\
&=
\frac{|a_1|+|a_2|}{1+|a_1||a_2|}
\end{align*}
Since
\[
\frac{|a_1|+|a_2|}{1+|a_1||a_2|}
>1
\iff
(1-|a_1|)(1-|a_2|)<0,
\]
it follows that $|a_3|>1$. Thus the diagonals intersect outside the circumcircle.
\end{proof}

\begin{corollary}
Let $a_3$ be defined as in \eqref{eq:a3} and the circle $\Gamma$ be defined as in Corollary \ref{cor:gamma}. Then $a_3 \in \Gamma$. 
\end{corollary}
\begin{proof}
\begin{align*}
    \left| a_3-\frac{a_1+a_2+a_3}{2} \right|
    &=\frac{1}{2}\left| \frac{a_1+a_2-(\overline{a_1}+\overline{a_2})a_1a_2}{1-|a_1|^2|a_2|^2}-(a_1+a_2)\right|\\
    &=\frac{1}{2}|a_1a_2|\left| \frac{-(\overline{a_1}+\overline{a_2})+(a_1+a_2)\overline{a_1}\overline{a_2}}{1-|a_1|^2|a_2|^2}\right|\\
    &=\frac{1}{2}|a_1a_2|\left|\overline{a_3}\right|\\
    &=\frac{1}{2}|a_1a_2a_3|
\end{align*}
\end{proof}

Recall that an \textit{orthoptic} of a curve is the locus of points at which a pair of tangents to the curve intersect at a right angle. The following is a well-known result:
\begin{theorem}\label{thm:orthoptic}
The orthoptic of a central conic, if it exists, is a circle.
\end{theorem}

In particular, the orthoptic of the conic $ \mathcal{D}$ in \eqref{eq:mathcalD} is the circle
\begin{equation*}
    x^2+y^2=a^2+\varepsilon b^2.
\end{equation*}

\begin{theorem}\label{thm:orthopticorfoci}
Let a cyclic quadrilateral $\mathcal Q$ be circumscribed about a central conic. Suppose that the circumcenter of $\mathcal Q$ coincides with the center of the conic. Then the circumcircle of the quadrilateral either passes through the foci of the conic or it is the orthoptic of the conic.
\end{theorem}
\begin{proof}
Assume that $\mathcal D$ is given by \eqref{eq:mathcalD}. Since the circumcenter of $\mathcal Q$ coincides with the center of $\mathcal D$, the distances from the center to the two foci are equal:
\[
d_+=d_-=\sqrt{a^2-\varepsilon b^2}.
\]
Substituting this into the generalized Cayley--Fuss relation \eqref{eq:cayleyfuss} gives
\[
(R^2-a^2+\varepsilon b^2)
(R^2-a^2-\varepsilon b^2)
(R^2+a^2-\varepsilon b^2)=0.
\]

Since
\[
a^2-\varepsilon b^2=c^2\ge0
\quad\text{and}\quad
R>0,
\]
the factor
\[
R^2+a^2-\varepsilon b^2=R^2+c^2
\]
is strictly positive and cannot vanish. Hence
\[
R^2=a^2\pm\varepsilon b^2.
\]

If
\[
R^2=a^2+\varepsilon b^2,
\]
then the circumcircle has equation
\[
x^2+y^2=a^2+\varepsilon b^2,
\]
which is precisely the orthoptic of $\mathcal D$ by Theorem~\ref{thm:orthoptic}.

On the other hand, if
\[
R^2=a^2-\varepsilon b^2=c^2,
\]
then the radius of the circumcircle equals the focal distance. Hence the circumcircle passes through the two foci of $\mathcal D$.

This completes the proof.
\end{proof}

\begin{theorem}
Let $Q$ be a cyclic quadrilateral circumscribed about a central conic $\mathcal D$ whose foci  neither lie on the circumcircle $\mathcal C$ of $Q$ nor the inverse points with respect to $\mathcal C$. If the circumcenter of $Q$ coincides with a focus of $\mathcal D$, then the anticenter of $Q$ coincides with the remaining focus of $\mathcal D$.
\end{theorem}
\begin{proof} Let $a_1,a_2$ denote the real foci (identified as complex points in $\mathbb C$) of $\mathcal D$. Then the diagonals of $\mathcal Q$ intersect at $a_3$ as given by \eqref{eq:a3}.

First assume $a_1=0$. Then $a_3=a_2$. Also note that the anticenter of $Q$ coincides with $a_2$ (Corollary \ref{cor:anticen}).

Conversely, let $a_3=a_2$. This gives   
\[
    a_1(1-|a_2|^2)(1-\overline{a_1}a_2)=0.
\]
Since $a_1,a_2 \notin \mathbb T$ or $a_1,a_2$ are not inverse points with respect to $\mathbb T$, the second and the third factors cannot be equal to 0. Therefore, $a_1=0$.
\end{proof}

\begin{theorem}\label{thm:orthodiag}
Let $Q$ be a cyclic quadrilateral circumscribed a central conic $\mathcal D$. Assume that the circumcenter of $Q$ coincides with a focus of $\mathcal D$. Then $Q$ is orthodiagonal. (See Figure \ref{fig:orthodiag}.)
\end{theorem}
\begin{proof} Since the diagonals of $Q$ intersect at its anticenter, $Q$ is either a rectangle or orthodiagonal. If $Q$ is a rectangle, then $\mathbb T$ must be the orthoptic of $\mathcal D$ and hence must be concentric with $\mathcal D$; a contradiction! 
\end{proof}

\begin{figure}
  \begin{subfigure}[b]{0.4\textwidth}
    \centering
\definecolor{wqwqwq}{rgb}{0.3764705882352941,0.3764705882352941,0.3764705882352941}
\begin{tikzpicture}[scale=2.5]
\clip(-1.2,-1.2) rectangle (1.2,1.2);
\draw[line width=1.pt,color=wqwqwq,fill=wqwqwq,fill opacity=0.10000000149011612] (0.5661784727134428,0.3220018138066305) -- (0.4941766589068122,0.3881802865200733) -- (0.4279981861933695,0.3161784727134428) -- (0.5,0.25) -- cycle; 
\draw [line width=1.pt,gray] (0.,0.) circle (1.cm);
\draw [rotate around={26.56505117707799:(0.25,0.125)},line width=1.pt,color=red] (0.25,0.125) ellipse (0.649519053431628cm and 0.5863019706351992cm);
\draw [line width=1.pt,blue] (-0.27427323751504007,0.9616518035042717)-- (-0.5167014952387196,-0.8561656176336863);
\draw [line width=1.pt,blue] (-0.5167014952387196,-0.8561656176336863)-- (0.9813184436212167,-0.192390520059677);
\draw [line width=1.pt,blue] (0.9813184436212167,-0.192390520059677)-- (0.8096562865988393,0.5869043342581297);
\draw [line width=1.pt,blue] (0.8096562865988393,0.5869043342581297)-- (-0.27427323751504007,0.9616518035042717);
\draw [line width=1.pt,brown] (-0.27427323751504007,0.9616518035042717)-- (0.9813184436212167,-0.192390520059677);
\draw [line width=1.pt,brown] (0.8096562865988393,0.5869043342581297)-- (-0.5167014952387196,-0.8561656176336863);
\draw [line width=1.pt,qqwuqq] (0.3535226030530883,0.38463064172229733)-- (0.14647739568005985,-0.13463064168777827);
\begin{scriptsize}
\draw[color=black] (-0.5827053148192841,0.9024101337733431) node {$\mathbb T$};
\draw [fill=black] (0.,0.) circle (0.7pt);
\draw[color=black] (0.041962958818530266,0.09794803967889065) node {$a_{1}$};
\draw [fill=black] (0.5,0.25) circle (0.7pt);
\draw[color=black] (0.5075828528216244,0.15) node {$a_{2}$};
\draw[color=red] (-0.14935610653917675,0.65) node {$\mathcal D$};
\draw [fill=wewdxt] (-0.27427323751504007,0.9616518035042717) circle (0.7pt);
\draw[color=wewdxt] (-0.29,1.07) node {$z_1$};
\draw [fill=wewdxt] (-0.5167014952387196,-0.8561656176336863) circle (0.7pt);
\draw[color=wewdxt] (-0.56,-0.95) node {$z_2$};
\draw [fill=wewdxt] (0.9813184436212167,-0.192390520059677) circle (0.7pt);
\draw[color=wewdxt] (1.08,-0.21) node {$z_3$};
\draw [fill=wewdxt] (0.8096562865988393,0.5869043342581297) circle (0.7pt);
\draw[color=wewdxt] (0.88,0.66) node {$z_4$};
\draw [fill=wewdxt] (0.3535226030530883,0.38463064172229733) circle (0.7pt);
\draw[color=wewdxt] (0.37,0.48) node {$L$};
\draw [fill=wewdxt] (0.14647739568005985,-0.13463064168777827) circle (0.7pt);
\draw[color=wewdxt] (0.17,-0.22) node {$M$};
\end{scriptsize}
\end{tikzpicture}
    \caption{$\mathcal D$ is an ellipse}
    \label{fig:orthodiag(A)}
\end{subfigure}
  \begin{subfigure}[b]{0.4\textwidth}
    \centering
\definecolor{qqwuqq}{rgb}{0.,0.4,0.}
\definecolor{zzttqq}{rgb}{0.6,0.2,0.}
\definecolor{wqwqwq}{rgb}{0.3764705882352941,0.3764705882352941,0.3764705882352941}
\definecolor{ffqqqq}{rgb}{1.,0.,0.}
\begin{tikzpicture}[scale=2.5]
\clip(-1.4,-1.4) rectangle (2.0,1.7);
\draw[line width=1.pt,color=wqwqwq,fill=wqwqwq,fill opacity=0.25] (0.9,0.53) -- (0.87,0.43) -- (0.97,0.39) -- (1.,0.5) -- cycle; 
\draw [line width=1.pt,gray] (0.,0.) circle (1.cm);
\draw [samples=50,domain=-0.99:0.99,rotate around={26.56505117707799:(0.5,0.25)},xshift=0.5cm,yshift=0.25cm,line width=1.pt,color=ffqqqq] plot ({0.4330127024391393*(1+(\x)^2)/(1-(\x)^2)},{0.35355338992343643*2*(\x)/(1-(\x)^2)});
\draw [samples=50,domain=-0.99:0.99,rotate around={26.56505117707799:(0.5,0.25)},xshift=0.5cm,yshift=0.25cm,line width=1.pt,color=ffqqqq] plot ({0.4330127024391393*(-1-(\x)^2)/(1-(\x)^2)},{0.35355338992343643*(-2)*(\x)/(1-(\x)^2)});
\draw [line width=1.pt,color=blue] (-0.3452235151888288,0.9385204976774181)-- (0.5689224688951009,-0.8223911626387428);
\draw [line width=1.pt,color=blue] (0.5689224688951009,-0.8223911626387428)-- (0.8320088813277284,0.5547623107167446);
\draw [line width=1.pt,color=blue] (0.8320088813277284,0.5547623107167446)-- (0.9442921642269864,0.3291083538585951);
\draw [line width=1.pt,color=blue] (0.9442921642269864,0.3291083538585951)-- (-0.3452235151888288,0.9385204976774181);
\draw [line width=1.pt,brown] (-0.3452235151888288,0.9385204976774181)-- (1.,0.5);
\draw [line width=1.pt,brown] (0.5689224688951009,-0.8223911626387428)-- (1.,0.5);
\draw [line width=1.pt,color=qqwuqq] (0.2433926830694498,0.7466414041970814)-- (0.7566073165610436,-0.24664140439007384);
\begin{scriptsize}
\draw [fill=black] (0.,0.) circle (0.7pt);
\draw[color=black] (-0.08,-0.08) node {$a_1$};
\draw[color=black] (-0.6631196061722898,0.8531769825074987) node {$\mathbb T$};
\draw [fill=black] (1.,0.5) circle (0.7pt);
\draw[color=black] (1.1,0.54) node {$a_2$};
\draw[color=red] (0.901715924202119,1.163045404363817) node {$\mathcal D$};
\draw [fill=wqwqwq] (-0.3452235151888288,0.9385204976774181) circle (0.7pt);
\draw[color=wqwqwq] (-0.3687446054087872,1.0236046145284736) node {$z_1$};
\draw [fill=wqwqwq] (0.5689224688951009,-0.8223911626387428) circle (0.7pt);
\draw[color=wqwqwq] (0.6,-0.92) node {$z_2$};
\draw [fill=wewdxt] (0.8320088813277284,0.5547623107167446) circle (0.7pt);
\draw[color=wewdxt] (0.92,0.61) node {$z_3$};
\draw [fill=wewdxt] (0.9442921642269864,0.3291083538585951) circle (0.7pt);
\draw[color=wewdxt] (1.05,0.37) node {$z_4$};
\draw [fill=wewdxt] (0.2433926830694498,0.7466414041970814) circle (0.7pt);
\draw[color=wewdxt] (0.28197908048948184,0.8325190877170775) node {$L$};
\draw [fill=wewdxt] (0.7566073165610436,-0.24664140439007384) circle (0.7pt);
\draw[color=wewdxt] (0.85,-0.25) node {$M$};
\end{scriptsize}
\end{tikzpicture}
    \caption{$\mathcal D$ is a hyperbola}
    \label{fig:orthodiag(B)}
\end{subfigure}
    \caption{Illustration of Theorem \ref{thm:orthodiag}. The circumcenter of the quadrilateral with ordered vertices $z_1,z_2,z_3,z_4$ is orthodiagonal.}
    \label{fig:orthodiag}
\end{figure}
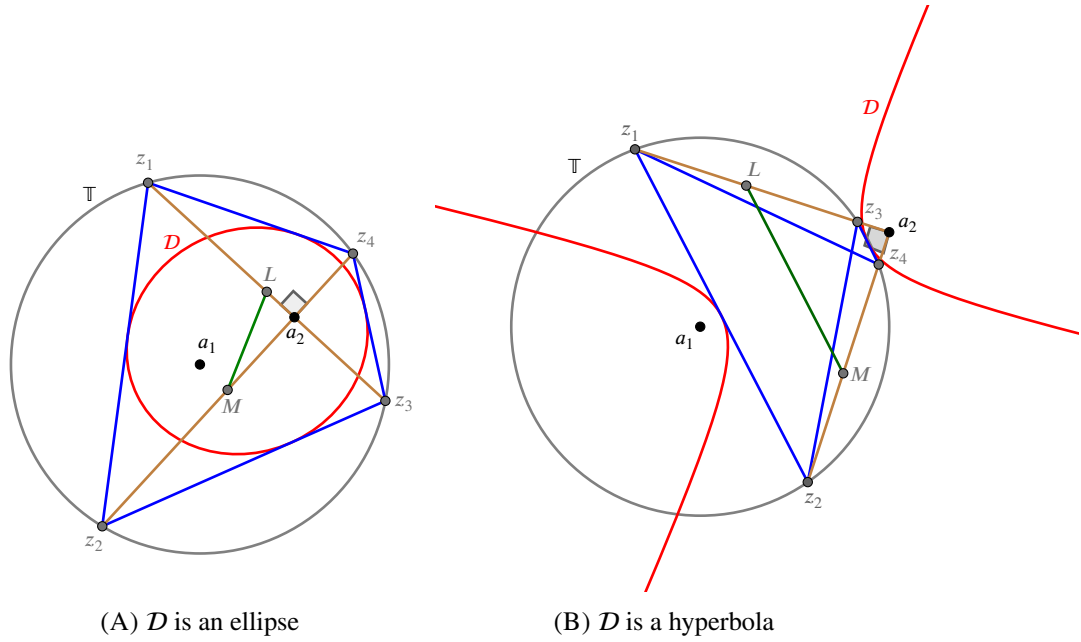

For a classical account of the geometric properties of cyclic orthodiagonal quadrilaterals, see \cite{AltshillerCourt1952}. The following corollaries are immediate consequences of Theorem \ref{thm:orthodiag}.

\begin{lemma}\label{lem:3}
Let $z_1,z_2,z_3,z_4 \in \mathbb T$ be the vertices of a cyclic quadrilateral $Q$ circumscribed about a central conic $\mathcal D$. Assume that the circumcenter of $Q$ coincides with a focus $f \in \mathbb C$ of $\mathcal D$. Then
\[
|z_1-f|^2+|z_2-f|^2+|z_3-f|^2+|z_4-f|^2=4.
\]
\end{lemma}
\begin{proof}
\begin{align*}
|z_1-f|^2+|z_2-f|^2+|z_3-f|^2+|z_4-f|^2&=4-2\operatorname{Re}(\overline{\sigma_1}f)+4|f|^2\\
&=4-4\operatorname{Re}(\overline{f}f)+4|f|^2\\
&=4.
\end{align*}
This proves the claim.
\end{proof}

\begin{corollary}\label{cor:sumsqside}
Let $Q$ be the cyclic quadrilateral as in Theorem \ref{thm:orthodiag}. Then the sum of the squares of the pair of opposite sides of $Q$ remains invariant throughout $\mathcal P$.
\end{corollary}

\begin{proof}
Since $ABCD$ is orthodiagonal, $\triangle AF_-B$ and $\triangle CF_-D$ are right-angled, giving
\[
|AB|^2+|CD|^2=|AF_-|^2+|BF_-|^2+|CF_-|^2+|DF_-|^2=4.
\]
Similarly, $|BC|^2+|DA|^2=4$, and the result follows.

Consequently,
\begin{equation}\label{eq:sqsumside}
|AB|^2+|BC|^2+|CD|^2+|DA|^2=8.
\end{equation}
\end{proof}

\begin{corollary}\label{cor:sumsqdiag}
Let $Q$ be the cyclic quadrilateral as in Theorem \ref{thm:orthodiag}. Then the sum of the squares of the diagonals of $Q$ equals $16a^2$, where $a$ is the length of the semi-major axis of the conic.
\end{corollary}

\begin{proof}
Applying Theorem \ref{thm:sumsq} with $n=4$, we obtain
\[
|AB|^2+|BC|^2+|CD|^2+|DA|^2+|AC|^2+|BD|^2=16-4|f|^2.
\]
Using \eqref{eq:sqsumside}, it follows that
\[
|AC|^2+|BD|^2=8-4|f|^2=16a^2.
\]
where the last equality follows from Theorem \ref{thm:ellorhyp4}.
\end{proof}

\begin{corollary}
Let $Q$ be the cyclic quadrilateral as in Theorem \ref{thm:orthodiag}. Then the distance between the midpoints of the diagonals of $Q$ equals the distance between the foci of the conic.
\end{corollary}

\begin{proof}
Let $L$ and $M$ denote the midpoints of the diagonals $AC$ and $BD$ of $ABCD$. By Euler's theorem for quadrilaterals,  
\[
|AB|^2+|BC|^2+|CD|^2+|DA|^2 = |AC|^2 + |BD|^2 + 4|LM|^2.
\]  
Using Corollaries \ref{cor:sumsqside} and \ref{cor:sumsqdiag}, we deduce  
\[
|LM| = |a_2|,
\]  
so that $|LM|=|a_2-a_1|$. Hence, the distance between the midpoints of the diagonals is invariant over $\mathcal{P}$ and equals the distance between the foci of $\mathcal{D}$.
\end{proof}

\begin{corollary}\label{cor:21}
Let $Q$ be the cyclic quadrilateral as in Theorem \ref{thm:orthodiag}. Then the sum of the squares of the distances from the circumcenter to any pair of opposite sides of $Q$ is equal to the square of its circumradius.
\end{corollary}

\begin{proof}
Let $Q=ABCD$ and without loss of generality, let the circumcenter of $Q$ be $F_+$. By Theorem \ref{thm:orthodiag}, $Q$ is orthodiagonal.  

Consider a pair of opposite sides, say $AB$ and $CD$. By Brahmagupta's theorem for orthodiagonal quadrilaterals,  
\[
\operatorname{dist}(F_+,AB)=\frac{1}{2}|CD|, \qquad 
\operatorname{dist}(F_+,CD)=\frac{1}{2}|AB|.
\]  
Using Corollary \ref{cor:sumsqdiag}, the sum of the squares of these distances satisfies  
\[
\operatorname{dist}(F_+,AB)^2 + \operatorname{dist}(F_+,CD)^2 = \frac{1}{4}(|AB|^2 + |CD|^2) = 1,
\]  
showing the claimed invariance over $\mathcal{P}$.
\end{proof}

\begin{corollary}
Let $ABCD$ be the cyclic quadrilateral as in Theorem \ref{thm:orthodiag}. For any pair of adjacent interior angles $\phi,\varphi$ of $ABCD$, one has
\begin{equation}\label{eq:2.16}
\sin^2 \phi + \sin^2 \varphi = (2a)^2,
\end{equation}
where $a$ is the length of the semi-major axis of $\mathcal{D}$.
\end{corollary}

\begin{proof}
By the law of sines for $\triangle ABC$ and $\triangle ABD$,  
\[
\frac{|BD|}{\sin A} = \frac{|AC|}{\sin B} = 2,
\]  
where $ABCD$ has unit circumradius. 

Using Corollary \ref{cor:sumsqdiag}, the sum of squares of sines of two adjacent angles satisfies  
\begin{equation}\label{eq:sinasinb}
    \sin^2 A + \sin^2 B = 4a^2
\end{equation}  
where $a$ is the length of the semi-major axis of $\mathcal{D}$. 

This completes the proof.
\end{proof}


\noindent \textbf{Disclosure Statement}. The author reports there is no conflict of interest.
\vspace{0.3cm}

\noindent
\textbf{Funding Declaration.} The author received no financial support for the research, authorship, and/or publication of this article.

\vspace{0.3cm}
\noindent
\textbf{Data Availability Statement.} Data sharing is not applicable to this article as no datasets were generated or analyzed during the current study.

\vspace{0.3cm}
\noindent
\thanks{\textbf{Acknowledgments}. The author is grateful to his wife, Saba Fatema, for her insightful discussions and suggestions.

\bibliographystyle{amsplain} 
\bibliography{references}

@article{Cayley1853a,
    author = {Cayley, A.},
    title = {{XII}. {N}ote on the {P}orism of the in-and-circumscribed {P}olygon},
    journal = {Philosophical Magazine},
    volume  = {6},
    number  = {37},
  year    = {1853},
  pages   = {99--102}
}

@article{Cayley1853b,
    author = {Cayley, A.},
    title = {{LVIII}. {C}orrection of two {T}heorems relating to the {P}orism of the in-and-circumscribed {P}olygon},
    journal = {Philosophical Magazine},
    volume  = {6},
    number  = {6},
  year    = {1853},
  pages   = {376--377}
}

@book{AltshillerCourt1952,
  author    = {Altshiller-Court, N.},
  title     = {College {G}eometry: {A}n {I}ntroduction to the {M}odern {G}eometry of the {T}riangle and the {C}ircle},
  publisher = {Dover Publications},
year = {2007},
note = {Dover reprint of the 1952 edition}
}

@article{Marden1945,
  author  = {Marden, M.},
  title   = {A {N}ote on the {Z}eros of the {S}ections of a {P}artial {F}raction},
  journal = {Bulletin of the American Mathematical Society},
  volume  = {51},
  number  = {12},
  year    = {1945},
  pages   = {935--940},
  note     = {\text{doi}:\href{http://dx.doi.org/10.1090/S0002-9904-1945-08470-5}{10.1090/S0002-9904-1945-08470-5}}
}

@article{Helmanetal.2022,
  author  = {Helman, M. and Laurain, D. and Garcia, R. and Reznik, D.},
  title   = {Poncelet {T}riangles: {A} {T}heory for {L}ocus {E}llipticity},
  journal = {Beitr Algebra Geom},
  year    = {2022},
  volume  = {63},
  number  = {3},
  pages   = {445--457},
  note    = {\text{doi}:\href{http://dx.doi.org/10.1007/s13366-021-00620-0}{10.1007/s13366-021-00620-0}}
}

@article{Helmanetal.2023,
  author  = {Helman, M. and Laurain, D. and Garcia, R. and Reznik, D.},
  title   = {Invariant {C}enter {P}ower and {E}lliptic {L}oci of {P}oncelet {T}riangles},
  journal = {J. Dyn. Control Syst.},
  year    = {2023},
  volume  = {29},
  number  = {1},
  pages   = {157--184},
  note    = {\text{doi}:\href{http://dx.doi.org/10.1007/s10883-021-09580-z}{10.1007/s10883-021-09580-z}}
}

@article{daepp2002ellipses,
  title     = {Ellipses and {F}inite {B}laschke {P}roducts},
  author    = {Daepp, U. and Gorkin, P. and Mortini, R.},
  journal   = {The American Mathematical Monthly},
  volume    = {109},
  number    = {9},
  pages     = {785--795},
  year      = {2002},
  publisher = {Taylor \& Francis},
  note    = {\text{doi}:\href{http://dx.doi.org/10.1080/00029890.2002.11919914}{10.1080/00029890.2002.11919914}}
}

@book{DaeppGorkinshaffervoss2018,
  author    = {U. Daepp and P. Gorkin and A. Shaffer and K. Voss},
  title     = {Finding {E}llipses: What {B}laschke {P}roducts, {P}oncelet’s {T}heorem, and the {N}umerical {R}ange {K}now about {E}ach {O}ther},
  publisher = {American Mathematical Society},
  year      = {2018},
  series    = {Mathematical Association of America Textbooks},
  address   = {Providence, RI},
  isbn      = {978-1-4704-4178-1}
}

@article{Garciaetal.2026,
  author  = {Garcia, R. and Helman, M. and Reznik, D.},
  title   = {Graceful {L}oci of {P}oncelet {T}riangles about the {I}ncircle and their {D}egeneracies},
  journal = {Beitr Algebra Geom},
  year    = {2026},
  volume  = {},
  number  = {},
  pages   = {},
  note    = {\text{doi}:\href{https://doi.org/10.1007/s13366-026-00825-1}{10.1007/s13366-026-00825-1}}
}

@article{Dragovic-Radnovic2024,
  author  = {Dragovi\'{c}, V. and Radnovi\'{c}, M.},
  title   = {Isoperiodic {F}amilies of {P}oncelet {P}olygons {I}nscribed in a {C}ircle and {C}ircumscribed about {C}onics from a {C}onfocal {P}encil},
  journal = {Geom Dedicata},
  year    = {2024},
  volume  = {218},
  pages  = {81:1--23},
note =         {\text{doi}:\href{http://dx.doi.org/10.1007/s10711-024-00929-9}{10.1007/s10711-024-00929-9}}
}

@article{Dragovic-Murad2026,
  author  = {Dragovi\'{c}, V. and Murad, M. H.},
  title   = {Generalized {C}happle--{E}uler {R}elation},
  year = {2026},
note =         {to appear in European Journal of Mathematics \text{arXiv}:\href{https://arxiv.org/abs/2603.00001v3}{2603.00001v3}}
}

@article{Murad2026b,
  author  = {M. H. Murad},
  title   = {A {V}ector {G}eneralization of {E}uler's {Q}uadrilateral {T}heorem},
 note =  {arXiv:\href{https://arxiv.org/abs/2603.15657}{2603.15657}}
}

@article{Fujimura2013,
  author    = {M. Fujimura},
  title     = {Inscribed {E}llipses and {B}laschke {P}roducts},
  journal   = {Computational Methods and Function Theory},
  volume    = {13},
  number    = {4},
  pages     = {557--573},
  year      = {2013},
  note       = {
  {doi:\href{https://doi.org/10.1007/s40315-013-0037-8}{10.1007/s40315-013-0037-8}}}
}

@book{Hahn2019,
  author    = {Hahn, L.-S.},
  title     = {Complex {N}umbers and {G}eometry},
  series    = {AMS/MAA Textbooks},
  volume    = {52},
  publisher = {American Mathematical Society},
  address   = {\,\,Providence, RI},
  year      = {2019}
}

@article{Celiketal2026,
  author  = {Çelik, M. and Duguin, M. and Guo, J. and Luo, D. and Spinelli, K. and
             Zeytuncu, Y. and Zhu, Z.},
  title   = {Exploring a {G}eometric {C}onjecture, {S}ome {P}roperties of {B}laschke {P}roducts,
             and the {G}eometry of {C}urves Formed by Them},
  journal = {Computational Methods and Function Theory},
  volume   = {26},
  number   = {1},
  pages    = {183--198},
  year     = {2026},
  note    = {\text{doi}:\href{ https://doi.org/10.1007/s40315-025-00579-2}{10.1007/s40315-025-00579-2}},
  doi      = {10.1007/s40315-025-00579-2},
  publisher = {Springer}
}
\end{document}